\documentclass[a4paper,12pt]{article}
\usepackage[left=2cm,right=2cm, top=2cm,bottom=3cm,bindingoffset=0cm]{geometry}
\usepackage{cmap}
\usepackage[T2A]{fontenc}
\usepackage[cp1251]{inputenc}
\usepackage{verbatim}
\usepackage{amsmath}
\usepackage{amsthm}
\usepackage{amssymb}
\usepackage{delarray}
\usepackage{cite}
\usepackage{hyperref}
\usepackage{mathrsfs}
\usepackage{tikz}
\usepackage{dsfont}
\usetikzlibrary{patterns}
\usepackage{caption}
\newcommand{\la}{\lambda}

\DeclareCaptionLabelSeparator{dot}{. }
\theoremstyle{plain}

\numberwithin{equation}{section}

\newtheorem{thm}{Theorem}[section]
\newtheorem{lem}[thm]{Lemma}

\theoremstyle{definition}

\newtheorem{ip}[thm]{Inverse Problem}
\newtheorem{example}[thm]{Example}

\theoremstyle{remark}

 \allowdisplaybreaks

\begin{document}
\begin{center}
{\large\bf Uniform stability of the inverse Sturm-Liouville problem\\[0.1cm] with polynomials in a boundary condition}
\\[0.3cm]
{\bf Bondarenko N.P., Chitorkin E.E.} \\[0.2cm]
\end{center}

\vspace{0.5cm}

{\bf Abstract.} This paper deals with the Sturm-Liouville problem with singular potential of the Sobolev space $W_2^{-1}$ and polynomials of the spectral parameter in a boundary condition. We prove the uniform boundedness and the uniform stability for the inverse spectral problem in the general non-self-adjoint case.
It is remarkable that our stability estimates are valid for some cases with different degrees of the polynomials for two compared operators. 

\medskip


\medskip


\vspace{1cm}

\section{Introduction} \label{sec:intro}

In this paper, we consider the following boundary value problem $L = L(\sigma, r_1, r_2)$:
\begin{gather} \label{eqv1}
-y'' + q(x) y = \lambda y, \quad x \in (0, \pi), \\ \label{bc1}
y^{[1]}(0) = 0, \quad r_1(\la)y^{[1]}(\pi) + r_2(\la)y(\pi) = 0,
\end{gather}
where $q(x)$ is a complex-valued distribution potential of the class $W_2^{-1}(0,\pi)$, that is, $q = \sigma'$, $\sigma \in L_2(0,\pi)$, $\la$ is the spectral parameter, $r_1(\la)$ and $r_2(\la)$ are polynomials, $y^{[1]} := y' - \sigma(x) y$ is the so-called quasi-derivative.
 
We aim to prove the uniform boundedness and the uniform stability of an inverse spectral problem for \eqref{eqv1}--\eqref{bc1}.
  

Inverse problems of spectral analysis involve reconstructing a differential operators (e.g., Sturm--Liouville operators) from information about their spectra. Such kind of problems find applications in various fields of science and engineering, including quantum and classical mechanics, geophysics, electronics, chemistry, and nanotechnology. The most complete results in inverse spectral theory were obtained for the Sturm-Liouville operators with constant coefficients in the boundary conditions (see, e.g., the monographs \cite{Lev84, FY01, Mar11, Krav20} and references therein).

Spectral problems with eigenparameter-dependent boundary conditions arise in thermodynamics \cite{BHM15}, quantum mechanics \cite{Gran22}, string theory \cite{GKR16}, and other physical applications \cite{Ful99}.
Inverse problems with polynomial dependence on the spectral parameter in boundary conditions were investigated in \cite{Chu01, BindBr021, BindBr022, ChFr, FrYu, FrYu12, YangWei18, Gul19, Gul20-ams, Gul23, amp, LocSolv, ChitBondStabMult, Bond25}. Some studies focus on self-adjoint problems with Herglotz-Nevanlinna rational functions in boundary conditions \cite{BindBr021, BindBr022, YangWei18, Gul19, Gul20-ams, Bond25}, which can be reduced to polynomial cases. In contrast to those works, this paper is concerned with the non-self-adjoint case. A constructive solution for the inverse Sturm--Liouville problem with polynomial boundary conditions was developed by Freiling and Yurko \cite{FrYu, FrYu12} using the method of spectral mappings (see \cite{FY01}), but only for regular (integrable) potentials.

Recent research has extended spectral theory to differential operators with singular (distributional) coefficients (see, e.g., \cite{Gul19, SavShkal03, Sav10, SS13, SS14, Hryn03, Hryn11, Dj, Eck15, KVB24}). The method of spectral mappings was adapted for the Sturm-Liouville problem \eqref{eqv1}--\eqref{bc1} with distribution potential and polynomials in the boundary conditions in \cite{amp}.

Stability of inverse Sturm-Liouville problems was studied by many authors \cite{Sav10, SS13, Hryn11, Akt87, MW05, Bled12, BSY13, BondButTr17, BK19, MT19, XGY22, GMXA23}. We especially highlight the studies by Savchuk and Shkalikov \cite{Sav10, SS13} and by Hryniv \cite{Hryn11}, which establish the unconditional uniform stability. This means that the constants in stability estimates depend only on some lower and upper bounds for the spectral data and do not depend on the potentials. Such estimates appeared to be useful for studying the convergence of finite data approximations to the potential (see \cite{SS14}). However, there is a lack of uniform stability results for problems with eigenparameter-dependent boundary conditions. In \cite{LocSolv} and \cite{ChitBondStabMult} the authors proved local solvability and stability for the cases of simple and multiple eigenvalues, respectively. Nevertheless, the property of uniform stability provides stronger and more universal guarantees of the proximity of solutions. The uniform stability of the inverse Sturm-Liouville problem with rational Herglotz-Nevanlinna functions in boundary conditions was analyzed in \cite{Bond25}, but the approach of \cite{Bond25} has a number of limitations. In particular, it requires the self-adjointness of the problem and works for the class $\sigma \in W_2^{\alpha}[0,\pi]$, $\alpha > 0$, while the problem \eqref{eqv1}--\eqref{bc1} in this paper is non-self-adjoint and has a potential of higher singularity order: $\sigma \in L_2(0,\pi)$. Furthermore, the stability estimates in \cite{Bond25} are obtained for two boundary problems with the same degrees of the polynomials in the boundary conditions. The approach of this paper allows us to remove those limitations.

In this paper, we study the uniform boundedness and stability of the inverse Sturm-Liouville problem for \eqref{eqv1}--\eqref{bc1} in the non-self-adjoint case. The function $\sigma$ and the polynomials $r_1$, $r_2$ are recovered from the eigenvalues $\{\la_n\}_{n \ge 1}$ and the so-called weight numbers $\{ \alpha_n\}_{n \ge 1}$ being the residues of the Weyl function. 
Basing on the technique of our previous studies \cite{amp, LocSolv, BondR25}, we reduce this non-linear inverse problem to a linear equation (called the main equation) in the Banach space $m$ of bounded infinite sequences. A crucial role in our analysis is played by the operator $(E + \tilde H(x))$ participating in the main equation \eqref{main_eq}. We require that the operator $(E + \tilde H(x))$ is invertible on $m$ and that the inverse operator $(E + \tilde H(x))^{-1}$ is uniformly bounded. This requirement actually is a condition on the spectral data, since the operator $\tilde H(x)$ is constructed by using only $\{  \la_n, \alpha_n\}_{n \ge 1}$. 
Assuming the uniform boundedness of the inverse operator, we obtain the uniform boundedness and, subsequently, the unconditional uniform stability for the inverse spectral problem with polynomials in the boundary condition. A feature of our stability theorem is the validity for some cases with different degrees of polynomials for two problems being compared. We present some examples of this phenomenon and discuss some limitations, which are related to non-solvability of the main equation. 

The paper is organized as follows. In Section~\ref{sec:main}, we formulate the main results of this paper. 
Section~\ref{sec:bound} presents the construction of the main equation and the proof of the uniform boundedness for the inverse spectral problem. In Section~\ref{sec:stab}, relying on the results of Section~\ref{sec:bound}, we investigate the uniform stability of recovering $\sigma$, $r_1$, and $r_2$.
In Section~\ref{sec:solvability} we get some non-solvability conditions for the main equation and, consequently, for the inverse problem. Section~\ref{sec:ex} contains two examples of stability for problems of form \eqref{eqv1}--\eqref{bc1} with different degrees of the polynomials.

\section{Main results} \label{sec:main}

In this section, we give some notations and formulate the main theorems of our paper.

For regularization of equation \eqref{eqv1} (see \cite{SavShkal03}), we represent it in the following equivalent form:
$$
-(y^{[1]})' - \sigma(x) y^{[1]} - \sigma^2 y = \lambda y, \quad x \in (0,\pi),
$$
and consider its solutions in the class of absolutely continuous functions $y \in AC[0,\pi]$ such that $y^{[1]} \in AC[0,\pi]$.

The polynomials in the right boundary condition of the problem $L$ have the form
$$
r_1(\la) = \sum\limits_{n=0}^{P_1} c_n\la^n, \quad r_2(\la) = \sum\limits_{n=0}^{P_2} d_n\la^n, \quad P_1, P_2 \ge 0.
$$

Denote $p = \max\{P_1, P_2\}$ and suppose that $P_1 \ge P_2$. The opposite case can be studied similarly.
Let us normalize the polynomials. Without loss of generality, suppose that $P_1 = P_2 = p$ and $c_p=1$. The coefficients $d_p, d_{p-1}, \dots$ can be equal zero, so the polynomial $r_2(\la)$ can actually have a lower degree than $r_1(\la)$. Denote the set of such polynomial pairs $(r_1, r_2)$ by $R_p$.

In \cite{amp} it is shown that spectrum $\{\la_n\}_{n \ge 1}$ of the problem \eqref{eqv1}-\eqref{bc1} is a countable set of eigenvalues. Let all of the eigenvalues be simple. Then they can be numbered according to their asymptotic behavior:
\begin{gather} \label{eigen_asymp}
\rho_n = \sqrt{\la_n} = n - p - 1 + \varkappa_n, \quad n \ge 1, \quad \{\varkappa_n\}_{n \ge 1} \in l_2.
\end{gather}

Next, define the Weyl solution $\Phi(x, \la)$ as the solution of equation \eqref{eqv1} with the boundary conditions:
$$
\Phi^{[1]}(0, \la) = 1, \quad r_1(\la) \Phi^{[1]}(\pi, \la) + r_2(\la) \Phi(\pi, \la) = 0,
$$
and the Weyl function of the problem $L$:
$$
M(\la) = \Phi(0, \la).
$$

According to \cite{amp}, the Weyl function is meromorphic in $\la$ and admits the representation
$$
M(\la) = \sum\limits_{n=1}^{\infty} \dfrac{\alpha_n}{\la - \la_n},
$$
where $\alpha_n \in \mathbb C$ are called the (generalized) weight numbers. They satisfy the asymptotics
\begin{equation} \label{weight_asymp}
\alpha_n = \dfrac{2}{\pi} + \kappa_n, \quad \{\kappa_n\} \in l_2. 
\end{equation}

We will call the set $S = \{\la_n, \alpha_n\}_{n \ge 1}$ the {\it spectral data} of the boundary value problem $L$. So, we can consider the following inverse problem:
\begin{ip}
\label{ip1}
Given the spectral data $S$, find $\sigma \in L_2(0,\pi)$ and $(r_1, r_2) \in R_p$.
\end{ip}

Along with the problem $L$, let us introduce the model problem $\tilde L = L(0, \la^p, 0)$. We agree that, if a symbol $\gamma$ denotes an object related to $L$, then the symbol $\tilde \gamma$ with tilde will denote the analogous object related to $\tilde L$. 
Note that the first eigenvalue $\lambda = 0$ of $\tilde L$ has multiplicity $(p + 1)$. So, we define the spectral data of $\tilde L$ as follows:
\begin{gather}
\label{model_sd}
    \tilde\lambda_n = \left\{ \begin{aligned} 
0, \quad 1 \le n \le p, \\
(n-p-1)^2, \quad n > p; 
\end{aligned} \right. \qquad 
\tilde\alpha_n = \begin{cases} 
\frac{1}{\pi}, & n = 1, \\
0, & 2 \le n \le p + 1, \\
\frac{2}{\pi}, & n > p + 1.
\end{cases} 
\end{gather}
It means that the set $\tilde S$ has the following form:
$$
\tilde S = \Bigg\{ \bigg(0, \frac{1}{\pi}\bigg), \underbrace{(0,0), \dots, (0,0)}_p \Bigg\} \bigcup \Bigg\{\bigg((n-p-1)^2, \frac{2}{\pi}\bigg)\Bigg\}_{n > p+1}.
$$

For an integer $p \ge 0$, denote by $\mathcal S_p$ the set of all collections $S = \{\la_n, \alpha_n\}_{n \ge 1}$ (not necessarily being the spectral data of some problem $L$) such that the asymptotics \eqref{eigen_asymp} and \eqref{weight_asymp} are satisfied for $\rho_n := \sqrt{\la_n}$ ($\arg\rho_n \in [-\frac{\pi}{2}, \frac{\pi}{2})$) and $\alpha_n$ with some $\{\varkappa_n\}, \{\kappa_n\} \in l_2$. In order to study the uniform boundedness and the uniform stability of Inverse Problem~\ref{ip1}, we have to introduce some subsets of $\mathcal S_p$. 
So, let $\mathcal{B}_\Omega$ be defined as follows:
\begin{equation} \label{defBO}
\mathcal{B}_\Omega = \Bigg\{ S\in \mathcal S_p\colon \sqrt{\sum_{n=1}^{\infty}\xi_n^2}\le \Omega \Bigg\},
\quad \xi_n = |\rho_n - \tilde\rho_n| + |\alpha_n - \tilde\alpha_n|.
\end{equation}

Next, for each $S \in \mathcal{S}_p$, we construct the operator $(E + \tilde H(x))$ in the Banach space $m$ of bounded infinite sequences according to Section~\ref{sec:bound}. Then, we can fix $\Omega > 0$ and some real $K > 0$ and introduce the following set:
$$
\mathcal{B}_{\Omega, K} = \{S \in \mathcal{B}_{\Omega} : \|(E + \tilde H(x))^{-1}\|_{m \to m} \le K\}.
$$

The first result of our work is the following theorem on uniform boundedness:

\begin{thm} \label{thm_uni_bound}
For each $S \in \mathcal{B}_{\Omega, K}$ there exist unique $\sigma(x) \in L_2(0, \pi)$ and $(r_1, r_2) \in R_p$, such that $S$ are the spectral data of the problem $L(\sigma, r_1, r_2)$. Moreover,
$$
\|\sigma(x)\|_{L_2(0, \pi)} \le C(\Omega, K), \quad |c_j|  \le C(\Omega, K), \quad |d_j | \le C(\Omega, K),\quad  j = \overline{0, p}.
$$
\end{thm}

Next, let us introduce two different problems $L^{(1)} = L(\sigma^{(1)}, r_1^{(1)}, r_2^{(1)})$ and $L^{(2)} = L(\sigma^{(2)}, r_1^{(2)}, r_2^{(2)})$ with the same parameter $p$. We agree that an element $\gamma^{(j)}$ is related to the corresponding problem $L^{(j)}$ ($j = 1, 2$).

Denote 
\begin{equation} \label{defZ}
\delta_n = |\rho_n^{(1)} - \rho_n^{(2)}| + |\alpha_n^{(1)} - \alpha_n^{(2)}|, \quad Z = \Big( \sum\limits_{n=1}^{\infty}\delta_n^2 \Big)^{\frac{1}{2}}. 
\end{equation}

Our next important result is the following theorem on uniform stability:

\begin{thm} \label{thm_uni_stab}
Suppose that $S^{(1)}, S^{(2)} \in \mathcal{B}_{\Omega, K}$. Then, the corresponding solutions $(\sigma^{(1)}, r_1^{(1)}, r_2^{(1)})$ and $(\sigma^{(2)}, r_1^{(2)}, r_2^{(2)})$ of Inverse Problem~\ref{ip1} satisfy the uniform estimate
\begin{equation} \label{uni}
\|\sigma^{(1)} - \sigma^{(2)}\|_{L_2[0, \pi]} + \sum\limits_{n=0}^{p-1} |c_n^{(1)} - c_n^{(2)}| + \sum\limits_{n=0}^p |d_n^{(1)} - d_n^{(2)}| \le C(\Omega, K)Z.
\end{equation}
\end{thm}

Note that the class $\mathcal S_p$ of spectral data contains collections $S = \{ \la_n, \alpha_n \}_{n \ge 1}$ with multiple eigenvalues and/or zero weight numbers. In fact, such cases correspond to polynomials $r_1(\la)$ and $r_2(\la)$, which are not relatively prime. After removing their common zeros, one arrives at the problem \eqref{eqv1}-\eqref{bc1} with polynomials of lower degrees. Our main results (Theorems~\ref{thm_uni_bound} and \ref{thm_uni_stab}) do not exclude this case. Thus, the uniform stability estimate \eqref{uni} can hold for problems $L^{(1)}$ and $L^{(2)}$ with different degrees of polynomials. This issue is discussed in more detail in Sections~\ref{sec:solvability} and~\ref{sec:ex}.

\section{Uniform boundedness} \label{sec:bound}

In this section, we prove Theorem~\ref{thm_uni_bound} on the uniform boundedness of Inverse Problem~\ref{ip1}. For this purpose, we reduce the inverse problem to a linear equation in the Banach space $m$ (the so-called main equation) and use the reconstruction formulas from \cite{amp} to get the uniform estimates for the solution. Note that our construction of the main equation \eqref{main_eq} differs from the previous studies \cite{amp, LocSolv}. Specifically, we apply the modification from \cite{BondR25} to make the components $\tilde H_{nk}(x)$ of the operator $\tilde H(x)$ to be continuous with respect to the spectral data of $L$. The latter feature is essential for investigation of the uniform stability. Furthermore, since we consider the singular potential $q \in W_2^{-1}(0,\pi)$, then obtaining the desired estimates requires more complicated technique, then in the regular case $q \in L_2(0,\pi)$. For the polynomials $r_1(\la)$ and $r_2(\la)$ in the boundary condition, we first get estimates on a finite set of specially chosen points and then apply the interpolation argument.

Let $\varphi(x, \la)$ be the solution of equation~\eqref{eqv1} with initial conditions $\varphi(0, \la) = 1$, $\varphi^{[1]}(0, \la) = 0$. Introduce the notations
\begin{gather*}
\la_{n0} = \la_n, \quad \rho_{n0} = \rho_n, \quad \alpha_{n0} = \alpha_n, \quad\la_{n1} = \tilde\la_n, \quad \rho_{n1} = \tilde\rho_n, \quad \alpha_{n1} = \tilde\alpha_n, \\
\varphi_{ni}(x) = \varphi(x, \la_{ni}), \quad \tilde\varphi_{ni}(x) = \tilde\varphi(x, \la_{ni}), \quad n \ge 1, \quad i = 0, 1,
\end{gather*}
and define the function
\begin{gather}
\label{D_def}
D(x, \la, \mu) = \int\limits_0^x {\varphi(t, \la)\varphi(t, \mu)} \, dt.
\end{gather}

Clearly, $\varphi(x, \la)$ is entire in $\la$ and $D(x, \la, \mu)$ is entire in $\la$ and $\mu$ for each fixed $x \in [0,\pi]$.

\begin{lem}
\label{prop_varphi}
The following relation holds:
\begin{align} \nonumber
\varphi(x, \la) = & \tilde \varphi(x, \la) - \sum\limits_{k=1}^{p+1}\alpha_{k0}\tilde D(x, \la, \la_{k0})\varphi_{k0}(x) + \dfrac{1}{\pi}\tilde D(x, \la, 0)\varphi(x, 0) \\ \label{servarphi}
& +\sum\limits_{k=p+2}^{\infty} \bigg( \alpha_{k1}\tilde D(x, \la, \la_{k1})\varphi_{k1}(x) - \alpha_{k0}\tilde D(x, \la, \la_{k0})\varphi_{k0}(x) \bigg),
\end{align}
where the series converges absolutely and uniformly with respect to $x \in [0, \pi]$ and $\la$ on compact sets.
\end{lem}

\begin{proof}
    Define the following regions:
    \begin{itemize}
        \item $\Xi = \{ \la = \rho^2: 0 < \Im \rho < \tau \}$, where $\tau$ is chosen so that $\la_{ni} \in \Xi$ for $n \ge 1$, $i=0,1$;
        \item $\zeta_N = \Big\{ \la : |\la| < \big(N+\frac{1}{2} \big)^2 \Big\}$, $N \in \mathbb N$;
    \end{itemize}
    and let $\Gamma_N$ be the border of the region $\Xi \cap \zeta_N$.

    Using \cite[Lemma~4.2]{amp}, we get that
    \begin{gather}
        \label{varphi_amp}
        \varphi(x, \la) = \tilde\varphi(x, \la) + \oint\limits_{\Gamma_N} \varphi(x, \mu)\tilde D(x, \la, \mu)\big( \tilde M(\mu) - M(\mu) \big) \, d\mu + \varepsilon_N(x, \la),
    \end{gather}
    where $\varepsilon_N(x, \la) \to 0$ as $N \to \infty$.

    Now we need to apply the Residue theorem. In view of \eqref{model_sd}, the Weyl function $\tilde M(\la)$ can be represented by the series
	\begin{equation} \label{seriesMt}
	\tilde M(\la) = \frac{1}{\pi \la} + \sum_{n = p+2}^{\infty} \frac{\alpha_{n1}}{\la- \la_{n1}}.
	\end{equation}

    Thus, the integrand in \eqref{varphi_amp} has only simple poles at $\la_{ni}$, $n \ge 1$, $i=0,1$. Consequently, we have
\begin{align*}
   & \mathop\mathrm{Res}\limits_{\mu=\la_{n0}} \varphi(x, \mu)\tilde D(x, \la, \mu)\big( \tilde M(\mu) - M(\mu) \big) = - \alpha_{n0}\tilde D(x, \la, \la_{n0}) \varphi_{n0}(x), \quad n\ge 1, \\
   & \mathop\mathrm{Res}\limits_{\mu=\la_{n1}} \varphi(x, \mu)\tilde D(x, \la, \mu)\big( \tilde M(\mu) - M(\mu) \big) = \alpha_{n1}\tilde D(x, \la, \la_{n1}) \varphi_{n1}(x), \quad n > p+1, \\
   & \mathop\mathrm{Res}\limits_{\mu=\la_{11}} \varphi(x, \mu)\tilde D(x, \la, \mu)\big( \tilde M(\mu) - M(\mu) \big) = \alpha_{11}\tilde D(x, \la, \la_{11}) \varphi_{11}(x) = \dfrac{1}{\pi} \tilde D(x, \la, 0)\varphi (x, 0).
\end{align*}
    
This together with \eqref{varphi_amp} yield the claim.
    
\end{proof}

Next, denote $\hat\rho_n = \rho_{n0} - \rho_{n1} = \rho_n - \tilde\rho_n$. For $x \in [0,\pi]$, introduce the following infinite column vectors $\varphi(x)$, $\tilde\varphi(x)$ and 
the infinite matrix $\tilde Q(x)$:
\begin{gather}
\notag
\varphi_n(x) = 
\begin{pmatrix}
\varphi_{n0}(x), & \varphi_{n1}(x)
\end{pmatrix}^T, \quad
\tilde \varphi_n(x) = 
\begin{pmatrix}
\tilde \varphi_{n0}(x), & \tilde \varphi_{n1}(x)
\end{pmatrix}^T,\\
\label{tildeQ_def}
\tilde Q_{ni, kj}(x) = \alpha_{kj}\tilde D(x, \rho_{ni}^2, \rho_{kj}^2), \quad
\tilde Q_{nk}(x) = 
\begin{pmatrix}
\tilde Q_{n0, k0}(x) & -\tilde Q_{n0, k1}(x) \\
\tilde Q_{n1, k0}(x) & -\tilde Q_{n1, k1}(x)
\end{pmatrix}, \\
\notag
\varphi(x) = \begin{pmatrix}
\varphi_{n}(x)
\end{pmatrix}_{n \ge 1}, \quad
\tilde\varphi(x) = \begin{pmatrix}
\tilde\varphi_{n}(x)
\end{pmatrix}_{n \ge 1} \quad
\tilde Q(x) =
\begin{pmatrix}
\tilde Q_{nk}(x)
\end{pmatrix}_{n, k \ge 1}.
\end{gather}

Then, due to Lemma~\ref{prop_varphi} we can get
\begin{equation} \label{relvvQ}
\tilde \varphi(x) = (E+\tilde Q(x))\varphi(x), 
\end{equation}
where $E$ is the identity operator.
Note that the series in \eqref{servarphi} converges only ``with brackets''.
To obtain absolute convergence, we need to modify equation \eqref{relvvQ} using the following transform induced by the block row vector $T = \big( T_n \big)_{n \ge 1}$ for $\hat\rho_n \neq 0$ defined as
\begin{gather}
\label{Tn_def}
T_n = 
\begin{pmatrix}
\hat\rho_n^{-1} & -\hat\rho_n^{-1} \\
0 & 1
\end{pmatrix}, \quad
T_n^{-1} = 
\begin{pmatrix}
\hat\rho_n & 1 \\
0 & 1
\end{pmatrix}.
\end{gather}
Furthermore, denote
\begin{gather*}
    \dot{\tilde\varphi}_{n1}(x) = \dfrac{\partial}{\partial\rho}\tilde\varphi(x, \rho^2)\bigg|_{\rho = \rho_{n1}}, \quad  \dot{\varphi}_{n1}(x) = \dfrac{\partial}{\partial\rho}\varphi(x, \rho^2)\bigg|_{\rho = \rho_{n1}},\\
    \dot{\tilde Q}_{n0, kj}(x) = \dfrac{\partial}{\partial\rho} \tilde Q(x, \rho^2, \la_{kj})\bigg|_{\rho=\rho_{n0}}.
\end{gather*}

Now introduce the next column vectors $\psi(x)$, $\tilde\psi(x)$ and the operator $\tilde H(x)$:
\begin{gather*}
\psi_n(x) = \begin{pmatrix}
\psi_{n0}(x)\\ \psi_{n1}(x)
\end{pmatrix} = \left\{ \begin{aligned} 
\begin{pmatrix}
\hat\rho_n^{-1} & -\hat\rho_n^{-1} \\
0 & 1
\end{pmatrix}
\begin{pmatrix}
\varphi_{n0}(x) \\
\varphi_{n1}(x)
\end{pmatrix} = T_n\varphi_n(x), \quad \hat\rho_n \neq 0, \\
\begin{pmatrix}
\dot\varphi_{n1}(x) \\
\varphi_{n1}(x)
\end{pmatrix}, \quad \hat\rho_n = 0,
\end{aligned} \right. \\
\tilde\psi_n(x) = \begin{pmatrix}
\tilde\psi_{n0}(x) \\
\tilde\psi_{n1}(x)
\end{pmatrix} = \left\{ \begin{aligned} 
\begin{pmatrix}
\hat\rho_n^{-1} & -\hat\rho_n^{-1} \\
0 & 1
\end{pmatrix}
\begin{pmatrix}
\tilde\varphi_{n0}(x) \\
\tilde\varphi_{n1}(x)
\end{pmatrix} = T_n\tilde\varphi_n(x), \quad \hat\rho_n \neq 0, \\
\begin{pmatrix}
\dot{\tilde\varphi}_{n1}(x) \\
\tilde\varphi_{n1}(x)
\end{pmatrix}, \quad \hat\rho_n = 0;
\end{aligned} \right. \\
\end{gather*}
\begin{align}
\notag
\tilde H_{nk}(x) = &
\begin{pmatrix}
\tilde H_{n0, k0}(x) & \tilde H_{n0, k1}(x) \\
\tilde H_{n1, k0}(x) & \tilde H_{n1, k1}(x)
\end{pmatrix} = \\ \label{Hnk}
& \left\{ \begin{aligned} 
\begin{pmatrix}
\hat\rho_n^{-1} & \hat\rho_n^{-1} \\
0 & 1
\end{pmatrix}
\begin{pmatrix}
\tilde Q_{n0, k0}(x) & -\tilde Q_{n0, k1}(x) \\
\tilde Q_{n1, k0}(x) & -\tilde Q_{n1, k1}(x)
\end{pmatrix}
\begin{pmatrix}
\hat\rho_k & 1 \\
0 & 1
\end{pmatrix} = T_n\tilde Q_{nk}(x)T_k^{-1}, \quad \hat\rho_n \neq 0, \\
\begin{pmatrix}
\dot{\tilde Q}_{n0, k0}(x) & -\dot{\tilde Q}_{n0, k1}(x) \\
\tilde Q_{n1, k0}(x) & -\tilde Q_{n1, k1}(x)
\end{pmatrix}T_k^{-1}, \quad \hat\rho_n = 0;
\end{aligned} \right. 
\end{align}
\begin{gather*}
\psi(x) = \begin{pmatrix}
\psi_{n}(x)
\end{pmatrix}_{n \ge 1}, \quad
\tilde\psi(x) = \begin{pmatrix}
\tilde\psi_{n}(x)
\end{pmatrix}_{n \ge 1} \quad
\tilde H(x) =
\begin{pmatrix}
\tilde H_{nk}(x)
\end{pmatrix}_{n, k \ge 1}.
\end{gather*}

Let $m$ be the Banach space of bounded infinite sequences $a = (a_{ni})_{n \ge 1, i = 0,1}$ with the norm $\| a \|_m = \sup |a_{ni}|$.
Using the above notations, we get from \eqref{relvvQ}, that for each fixed $x \in [0, \pi]$ the vector $\psi(x)$ satisfies the so-called main equation of Inverse Problem~\ref{ip1}:
\begin{gather}
\label{main_eq}
\tilde\psi(x) = (E+\tilde H(x))\psi(x)
\end{gather}
in the Banach space $m$. Also, the vectors $\tilde\psi(x)$ and $\psi(x)$ belong to $m$, and $\tilde H(x)$ is a bounded linear operator from $m$ to $m$.

Below, we use the main equation \eqref{main_eq} to obtain uniform estimates for several auxiliary functions and, finally, for the coefficients $\sigma(x)$, $r_1(\la)$, and $r_2(\la)$
of the problem $L$.

\begin{lem} \label{lem_tilde_est}
For $S \in \mathcal{B}_\Omega$ the following estimates hold:
\begin{enumerate}
    \item $|\tilde\varphi_{ni}(x)| \le C(\Omega)$,
    \item $|\tilde\varphi_{n0}(x) - \tilde\varphi_{n1}(x)| \le C(\Omega)\xi_n$,
    \item $|\tilde\varphi^{[1]}_{ni}(\pi)| \le C(\Omega)n\xi_n$,
    \item $|\tilde\varphi^{[1]}_{n0}(x) - \tilde\varphi^{[1]}_{n1}(x)| \le C(\Omega)n\xi_n$,
    \item $|\tilde\psi_{ni}(x)| \le C(\Omega)$,
    \item $|\tilde\psi^{[1]}_{ni}(x)| \le C(\Omega)n$,
    \item $|\tilde H_{ni, kj}(x)| \le \dfrac{C(\Omega)\xi_k}{|n-k|+1}$,
\end{enumerate}
where $n, k \ge 1$, $i, j = 0, 1$, $x \in [0, \pi]$. Then
$$
\|\tilde \psi(x)\|_m \le C(\Omega), \quad \|\tilde H(x)\|_{m \to m} \le C(\Omega), \quad x \in [0,\pi].
$$
\end{lem}

This lemma can be proved in the same way as \cite[Lemma~5.1]{BondR25}.

Let $T^{-1}$ operate with $f \in m$ according to the following rule:
\begin{gather*}
    T^{-1} f = \Big( T_n^{-1} f_n \Big)_{n \ge 1}, \\
 T_n^{-1} f_n =
\begin{pmatrix}
\hat\rho_n & 1 \\
0 & 1
\end{pmatrix}
\begin{pmatrix}
f_{n0} \\
f_{n1}
\end{pmatrix} = \begin{pmatrix}
\hat\rho_nf_{n0} + f_{n1} \\
f_{n1}
\end{pmatrix}.
\end{gather*}

\begin{lem} \label{lem_qt_oper}
For each $S \in \mathcal{B}_{\Omega}$ and $x \in [0, \pi]$, the operator $\tilde Q(x) T^{-1}$ is a mapping from $m$ to $l_2$. Moreover, for $f \in m$ the following estimate holds:
$$
\|\tilde Q(x) T^{-1} f\|_{l_2} \le C(\Omega)\|f\|_m.
$$
\end{lem}

\begin{lem} \label{lem_q_diff}
Let $f \in m$ and $q_{ni}(x) = \big( \tilde Q(x)T^{-1} f \big)_{ni}$, where $\tilde Q(x)$ is constructed for $S \in \mathcal{B}_{\Omega}$. Then, $\{(q_{n0} - q_{n1}) (x)\}_{n \ge 1} \in l_1$ and $\|\{(q_{n0} - q_{n1}) (x)\}\|_{l_1} \le C(\Omega)\|f\|_m$.
\end{lem}

The proofs of Lemmas~\ref{lem_qt_oper} and \ref{lem_q_diff} are analogous to the proofs of Lemmas~4.4 and 4.5 of \cite{LocSolv}, respectively. It can be shown that all the estimates throughout those proofs are uniform by $S \in \mathcal B(\Omega)$.

Thus, using the model problem $\tilde L$ and some complex numbers $S \in \mathcal{B}_{\Omega, K}$, we can build the column vector $\tilde\psi(x)$, the operator $\tilde H(x)$, and consequently the uniquely solvable equation of form \eqref{main_eq} in the Banach space $m$. Its solution can be found as $\psi(x) = (E+\tilde H(x))^{-1}\tilde\psi(x)$. Using elements of $\psi(x)$, we can find $\varphi(x) = T^{-1}\psi(x)$.

Denote $g(x) := \tilde\varphi(x) - \varphi(x)$. Then, we can get from \eqref{main_eq} that
$$
g(x) = T^{-1}\tilde H(x)\psi(x) = \tilde Q(x) T^{-1} \psi(x).
$$

\begin{lem} \label{lem_est}
For $S \in \mathcal{B}_{\Omega, K}$ and each fixed $x \in [0,\pi]$, there hold
\begin{enumerate}
    \item $|\psi_{ni}(x)| \le C(\Omega, K)$,
    \item $|g_{ni}(x)| \le C(\Omega, K)$,
    \item $\{ g_{ni}(x) \} \in l_2$ and $\|\{g_{ni}(x)\}\|_{l_2} \le C(\Omega, K)$,
    \item $\{ g_{n0}(x) - g_{n1}(x)\} \in l_1$ and $\|\{g_{n0}(x) - g_{n1}(x)\}\|_{l_1} \le C(\Omega, K)$,
\end{enumerate}
where $n \ge 1$, $i = 0, 1$.
\end{lem}

This lemma can be proved in the same way as \cite[Lemma~5.2]{LocSolv}. Moreover, it can be shown, that all the estimates are uniform by $S \in \mathcal{B}_{\Omega, K}$.

\begin{lem} 
\label{reconstruction}
For any complex numbers $S \in \mathcal B_{\Omega, K}$, Inverse Problem~\ref{ip1} has the unique solution, which can be found by the reconstruction formulas:
\begin{gather}
\label{rec_sigma}
\sigma(x) = \sum\limits_{k=1}^{\infty}\sum\limits_{j=0}^{1}(-1)^j\alpha_{kj}(1-2\tilde\varphi_{kj}(x)\varphi_{kj}(x)),\\
\label{rec_r1}
r_1(\la) = \prod\limits_{k=1}^{p}(\la-\la_{k0})\prod\limits_{k=p+1}^{\infty}\dfrac{\la-\la_{k0}}{\la-\la_{k1}}\Bigg( 1-\sum\limits_{k=1}^{\infty}\dfrac{\alpha_{k0}\tilde\varphi_{k0}^{[1]}(\pi)\varphi_{k0}(\pi)}{\la-\la_{k0}} \Bigg),\\
\notag
r_2(\la) = \prod\limits_{k=1}^{p}(\la-\la_{k0})\prod\limits_{k=p+1}^{\infty}\dfrac{\la-\la_{k0}}{\la-\la_{k1}}\Bigg( \sum\limits_{k=1}^{\infty}\dfrac{\alpha_{k0}\tilde\varphi_{k0}^{[1]}(\pi)\varphi_{k0}^{[1]}(\pi)}{\la-\la_{k0}}-\\
\label{rec_r2}
\sum\limits_{k=1}^{\infty}\sum\limits_{j=0}^{1}(-1)^j\alpha_{kj}(1-\tilde\varphi_{kj}(x)\varphi_{kj}(x)) \Bigg),
\end{gather}
where the series in \eqref{rec_sigma} converges in $L_2(0, \pi)$, 
the quasi-derivative $\varphi^{[1]}_{k0}$ is defined as $\varphi'_{k0} - \sigma \varphi_{k0}$ and is absolutely continuous on $[0,\pi]$,
the series and the infinite products in \eqref{rec_r1} and \eqref{rec_r2} converge absolutely and uniformly by $\la$ on compact sets excluding $\{ \la_{kj} \}$.
\end{lem}

\begin{proof}
It follows from $S \in B_{\Omega, K}$ that the asymptotic formulas \eqref{eigen_asymp} and \eqref{weight_asymp} hold and the main equation \eqref{main_eq} has a unique solution. Then, the existence for a solution of Inverse Problem~\ref{ip1} can be shown similarly to the arguments of \cite[Section~5]{LocSolv}. Namely, we construct the function $\sigma(x)$ by formula \eqref{rec_sigma} and prove the convergence of the series in $L_2(0,\pi)$. Next, we define the quasi-derivative $\varphi_{k0}^{[1]} := \varphi'_{k0} - \sigma \varphi_{k0}$, construct $r_1(\la)$ and $r_2(\la)$ via \eqref{rec_r1} and \eqref{rec_r2}, respectively. Furthermore, we show that $(r_1, r_2) \in R_p$ analogously to Theorem 2.8 in \cite{amp} and that the boundary value problem $L = L(\sigma, r_1, r_2)$ has the spectral data $S$. The uniqueness of the inverse problem solution is given by Theorem~2.3 in \cite{amp}.
\end{proof}

Let us prove the uniform boundedness of $\sigma$, $r_1$, and $r_2$.

\begin{lem} \label{lem_sigma_diff}
For each $S \in \mathcal{B}_{\Omega, K}$, we have $\|\sigma\|_{L_2(0, \pi)} \le C(\Omega, K)$.
\end{lem}

\begin{proof}
    The reconstruction formula \eqref{rec_sigma} can be represented in the following form:
    $$
    \sigma(x) = \sum_{j=1}^{7} S_j(x),
    $$
    where
    \begin{align*}
        &S_1(x) = \sum_k(\alpha_{k1} - \alpha_{k0})(2\cos^2{(k-p-1)x} - 1), \\
        &S_2(x) = \sum_k 2\alpha_{k0}\tilde\varphi_{k0}(x)(g_{k0}(x) - g_{k1}(x)), \\
        &S_3(x) = \sum_k 2(\alpha_{k0} - \alpha_{k1})g_{k1}(x)\tilde\varphi_{k0}(x),\\
        &S_4(x) = \sum_k 2(\alpha_{k0} - \alpha_{k1})(\tilde\varphi_{k1}(x) - \tilde\varphi_{k0}(x))(\tilde\varphi_{k1}(x) + \tilde\varphi_{k0}(x)),\\
        &S_5(x) = \sum_k 2\alpha_{k1}(\tilde\varphi_{k1}(x) - \tilde\varphi_{k0}(x))(\tilde\varphi_{k0}(x) - \tilde\varphi_{k1}(x)),\\
        &S_6(x) = \sum_k 2\alpha_{k1}g_{k1}(x)(\tilde\varphi_{k0}(x) - \tilde\varphi_{k1}(x)),\\
        &S_7(x) = \sum_k 4\alpha_{k1}\tilde\varphi_{k1}(x)(\tilde\varphi_{k1}(x) - \tilde\varphi_{k0}(x)).
    \end{align*}

    In view of \eqref{weight_asymp} and \eqref{defBO}, we have 
    \begin{equation} \label{estal}
       |\alpha_{kj}| \le C(\Omega), \quad |\alpha_{k1} - \alpha_{k0}| \le \xi_n,
    \end{equation}
    for any $S \in B_{\Omega}$. Recall that $\{ \xi_n \} \in l_2$. Consequently, the series $S_1(x)$ converges in $L_2(0,\pi)$ and $\|S_1(x)\|_{L_2(0, \pi)} \le C(\Omega)$.
    Furthermore, using \eqref{estal} together with estimates 1--2 of Lemma~\ref{lem_tilde_est} and estimates 3--4 of Lemma~\ref{lem_est}, we conclude
    that, for $j = \overline{2,6}$, the series $S_j(x)$ converge absolutely and $|S_j(x)| \le C(\Omega, K)$. The series $S_7(x)$ can be estimated analogously by recalling
    that $\tilde \varphi_{kj}(x) = \cos (\rho_{kj} x)$ and the asymptotics \eqref{eigen_asymp}. The sum of the resulting estimates yields the claim.      
\end{proof}

Since $(r_1, r_2) \in R_p$, then
$$
r_1(\la) = \sum\limits_{j=0}^{p}c_j\la^j, \quad r_2(\la) = \sum\limits_{j=0}^{p}d_j\la^j.
$$

\begin{lem} \label{lem_cj_diff}
For each $S \in \mathcal{B}_{\Omega, K}$, we have $|c_j| \le C(\Omega, K)$, $|d_j| \le C(\Omega, K)$, $j=\overline{0, p}$.
\end{lem}

\begin{proof}

    Let us prove the assertion of the lemma for $c_j$, $j=\overline{0, p}$.

    \textit{Step 1.}
    Consider the compact set 
    $$
       \Lambda_{\mathcal E, \varepsilon, S} := \bigl \{ \lambda \in \mathbb C \colon |\la| \le \mathcal E, \, |\la - \la_{ni}| \ge \varepsilon n^2, \, n \ge 1, \, i=0, 1 \bigr \},
       \quad \mathcal E, \, \varepsilon > 0.
    $$
    Let us prove the uniform estimate
    \begin{equation} \label{estr1}
       |r_1(\la)| \le C, \quad S \in \mathcal B_{\Omega, K}, \quad \la \in \Lambda_{\mathcal E, \varepsilon, S}.
    \end{equation}
    Here and below at this step of the proof, we denote by $C$ various positive constants depending on $\Omega, K, \mathcal E, \varepsilon$.
        
    As $|\la| \le \mathcal{E}$, then
    \begin{gather}
    \label{fin_est_step1}
    \bigg| \prod\limits_{k=1}^{p}(\la-\la_{k0}) \bigg| \le C.
    \end{gather}
    Next, let us investigate the second term of the reconstruction formula:
    $$
    \Bigg| \prod\limits_{k=p+1}^{\infty}\dfrac{\la-\la_{k0}}{\la-\la_{k1}} \Bigg| = \prod\limits_{k=p+1}^{\infty} \Bigg |1 + \dfrac{\la_{k1}-\la_{k0}}{\la-\la_{k1}} \Bigg| \le \prod\limits_{k=p+1}^{\infty} \Bigg( 1 + \dfrac{|\hat\rho_k|^2 + 2|\tilde\rho_k||\hat\rho_k|}{|\la-\la_{k1}|} \Bigg).
    $$
    Consider $\eta_k(\la) = \dfrac{|\hat\rho_k|^2 + 2|\tilde\rho_k||\hat\rho_k|}{|\la-\la_{k1}|}$. Using the inequality $|\hat\rho_k| \le \xi_k$ and the condition $|\la - \la_{k1}| \ge \varepsilon k^2$ we can get the estimate for $\eta_k(\la)$:
    $$
    \eta_k(\la) = \dfrac{|\hat\rho_k|^2 + 2|\tilde\rho_k||\hat\rho_k|}{|\la-\la_{k1}|} \le C \dfrac{\xi_k^2+2k\xi_k}{k^2} \le \dfrac{C\xi_k}{k}.
    $$
    Next, take the logarithm of the product:
    $$
    \ln \prod\limits_{k=p+1}^{\infty}(1+\eta_k(\la)) = \sum\limits_{k=p+1}^{\infty}\ln(1+\eta_k(\la)) \le \sum\limits_{k=p+1}^{\infty} \ln\bigg( 1 + \dfrac{C\xi_k}{k} \bigg) \le \sum\limits_{k=p+1}^{\infty} \bigg( \dfrac{C\xi_k}{k} + O\bigg( \dfrac{\xi_k^2}{k^2} \bigg) \bigg) \le C.
    $$
    Then, the infinite product converges and
    \begin{gather}
        \label{inf_est_step1}
        \Bigg| \prod\limits_{k=p+1}^{\infty}\dfrac{\la-\la_{k0}}{\la-\la_{k1}} \Bigg| \le C.
    \end{gather}
    
    Next, due to estimate 3 from Lemma~\ref{lem_tilde_est}, estimate 2 from Lemma~\ref{lem_est} and the condition $|\la - \la_{k0}| \ge \varepsilon k^2$ we obtain
    \begin{gather}
    \label{sum_est_step1}
        \Bigg| \sum\limits_{k=1}^{\infty}\dfrac{\alpha_{k0}\tilde\varphi_{k0}^{[1]}(\pi)\varphi_{k0}(\pi)}{\la-\la_{k0}} \Bigg| \le \sum\limits_{k=1}^{\infty} \dfrac{C\xi_k}{k} \le C.
    \end{gather}
    
    Substituting the estimates~\eqref{fin_est_step1}--\eqref{sum_est_step1} into \eqref{rec_r1}, we arrive at \eqref{estr1}.
    
    \medskip
    
    \textit{Step 2}. In order to get the estimates for the coefficients $c_j$ from \eqref{estr1}, we use the Lagrange interpolation. As we have a polynomial of degree $p$, then we need to choose $p+1$ points. At first, let us investigate some properties of $\la_{k0}$. Due to the asymptotics \eqref{eigen_asymp} we can get
    \begin{gather}
        \label{real_k0_est}
        \Re\la_{k0} \ge (k-p-1)^2-2|k-p-1|\Omega - \Omega^2 \ge -5\Omega^2.
    \end{gather}
	Fix $p+1$ values $\la_i$, such that $\la_i = -6\Omega^2-\vartheta_i$, where $\vartheta_i \in (0, \Omega^2)$. It means that $|\la_i| \le \mathcal{E}$, where $\mathcal{E} = \mathcal{E}(\Omega) = 7\Omega^2$. According to \eqref{real_k0_est}, we have:
    \begin{gather}
    \label{silly_est}
        |\la_i - \la_{k0}| \ge \Re(\la_{k0}-\la_i) \ge -5\Omega^2 + 6\Omega^2 + \vartheta_i \ge \Omega^2.    
    \end{gather}
    Next, we need to strengthen this estimate. Using \eqref{real_k0_est}, we obtain:
    \begin{align*}
    |\la_i - \la_{k0}| & \ge (k-p-1)^2 - 2|k-p-1|\Omega + 5\Omega^2 + \vartheta_i\\
    & \ge k^2-2k(p+1)+(p+1)^2 - 2k\Omega + C(\Omega) \\
    & = k^2f(k),
    \end{align*}
    where 
    $$
    f(k) = 1 - \dfrac{2(p+1+\Omega)}{k}+\dfrac{C(\Omega)}{k^2}. 
    $$
    Obviously, $f(k) \to 1$ as $k \to \infty$. Consequently, there exists $k^*$ such that, for any $k \ge k^*$, we have $|\la_i - \la_{k0}| \ge \dfrac{1}{2}k^2$. For $k < k^*$ we can use \eqref{silly_est}:
    $$
    |\la_i - \la_{k0}| \ge \Omega^2 \ge \dfrac{\Omega^2k^2}{(k^*)^2}.
    $$
    It means that $|\la_i - \la_{k0}| \ge \varepsilon k^2$, where 
    $$
    \varepsilon = \varepsilon(\Omega) := \min{\bigg\{ \dfrac{1}{2}, \dfrac{\Omega^2}{(k^*)^2} \bigg\}}.
    $$
    For $|\la_i - \la_{k1}|$, the analogous estimate can be proved. 
    Hence, for every $S \in \mathcal B_{\Omega, K}$, the chosen points $\la_i$ ($i = \overline{1,p+1}$) belong to $\Lambda_{\mathcal E, \varepsilon, S}$, where $\mathcal E =\mathcal E(\Omega)$ and $\varepsilon = \varepsilon(\Omega)$. Therefore, the estimate \eqref{estr1} implies
    $$
    	|r_1(\la_i)| \le C(\Omega, K), \quad i = \overline{1,p+1}.
    $$
    Applying the Lagrange interpolation formula by the points $\la_i$, we arrive at the assertion of the lemma for the coefficients $\{c_j\}_{j = 0}^p$ of the polynomial $r_1(\la)$. The proof for $\{ d_j \}_{j = 0}^p$ is analogous.
    \end{proof}

Lemmas~\ref{reconstruction}, \ref{lem_sigma_diff}, and \ref{lem_cj_diff} together imply Theorem~\ref{thm_uni_bound}.

\section{Uniform stability} \label{sec:stab}

In this section, we prove Theorem~\ref{thm_uni_stab} on the uniform stability of Inverse Problem~\ref{ip1}. 
A crucial role in our proofs is played by the propoerties of the operator $\tilde H(x)$ from the main equation \eqref{main_eq} and Theorem~\ref{thm_uni_bound} on the uniform boundedness.
At first, we deduce some properties of the operator 
\begin{equation} \label{defA}
\mathcal A := \big( \tilde Q^{(1)}(x)(T^{(1)})^{-1} - \tilde Q^{(2)}(x)(T^{(2)})^{-1} \big),
\end{equation}
and then use them with the reconstruction formulas to get the uniform stability estimates.

\begin{lem} \label{lem_tilde_est_12}
For $S^{(1)}, S^{(2)} \in \mathcal{B}_\Omega$ the following estimates hold:
\begin{enumerate}
    \item $|\tilde\varphi^{(1)}_{n0}(x) - \tilde\varphi^{(2)}_{n0}(x)| \le C(\Omega)\delta_n$,
    \item $|\tilde\psi^{(1)}_{n0}(x) - \tilde\psi^{(2)}_{n0}(x)| \le C(\Omega)\delta_n$,
    \item $\|\tilde H^{(1)}(x) - \tilde H^{(2)}(x)\|_{m \to m} \le C(\Omega)Z$,
\end{enumerate}
where $n \ge 1$, $x \in [0, \pi]$.
\end{lem}

This lemma is proved similarly to \cite[Lemma~6.1]{BondR25}.

\begin{lem} \label{lem_qt_12}
For each fixed $x \in [0, \pi]$ and $S^{(1)}, S^{(2)} \in \mathcal{B}_\Omega$, the operator $\mathcal A$, which is defined by \eqref{defA},
maps $m$ to $l_2$ and is uniformly bounded. Moreover, if $f \in m$, then
$$
\|(\mathcal A f)_{ni}\|_{l_2} \le C(\Omega)Z\|f\|_m.
$$
\end{lem}

\begin{proof}
Let $f \in m$. Then, due to \eqref{tildeQ_def} and \eqref{Tn_def}, there holds
$$
\tilde Q^{(j)}_{nk}(x)(T^{(j)}_k)^{-1}f_k = \begin{pmatrix}
\tilde Q^{(j)}_{n0,k0}(x)(\hat\rho^{(j)}_kf_{k0}+f_{k1}) - \tilde Q^{(j)}_{n0,k1}(x)f_{k1} \\
\tilde Q^{(j)}_{n1,k0}(x)(\hat\rho^{(j)}_kf_{k0}+f_{k1}) - \tilde Q^{(j)}_{n1,k1}(x)f_{k1}
\end{pmatrix}, \quad j = 1, 2, \quad n, k \ge 1.
$$

Using \eqref{D_def}, we get
\begin{multline} \label{intF1F2}
\bigg( \Big( \tilde Q^{(1)}(x)(T^{(1)})^{-1} - \tilde Q^{(2)}(x)(T^{(2)})^{-1} \Big) f \bigg)_{ni} \\ = \sum\limits_{n} \int\limits_{0}^{x} ( \tilde\varphi_{ni}^{(1)}(t) - \tilde\varphi_{ni}^{(2)}(t) ) F_1(t) \, dt + \sum\limits_{n} \int\limits_{0}^{x} \tilde\varphi_{ni}^{(2)}(t) F_2(t) \, dt,
\end{multline}
where
\begin{gather}
\label{F1_12}
F_1(t) = \sum\limits_{k} ( \hat\rho_k^{(1)}f_{k0}\alpha_{k0}^{(1)}\tilde\varphi_{k0}^{(1)}(t) + f_{k1} ( \alpha_{k0}^{(1)}\tilde\varphi_{k0}^{(1)}(t) - \alpha_{k1}^{(1)}\tilde\varphi_{k1}^{(1)}(t) ) ), \\
\label{F2_12}
F_2(t) = \sum\limits_{k}( f_{k0} ( \hat\rho_k^{(1)}\alpha_{k0}^{(1)}\tilde\varphi_{k0}^{(1)}(t) - \hat\rho_k^{(2)}\alpha_{k0}^{(2)}\tilde\varphi_{k0}^{(2)}(t) ) + f_{k1} ( \alpha_{k0}^{(1)}\tilde\varphi_{k0}^{(1)}(t) - \alpha_{k0}^{(2)}\tilde\varphi_{k0}^{(2)}(t) ) ). 
\end{gather}

Similarly to Lemma~\ref{lem_qt_oper}, we can get that $F_1(t) \in L_2(0, \pi)$. Using \eqref{defZ} and the estimate 1 of Lemma~\ref{lem_tilde_est_12} we can get that series in \eqref{F2_12} converges in $L_2(0, \pi)$. So, $F_2(t) \in L_2(0, \pi)$ and $\|F_2(t)\|_{L_2(0, \pi)} \le C(\Omega)Z\|f\|_m$. This yields the claim.

\end{proof}

\begin{lem} \label{lem_diff_qt_12}
Let $f \in m$ and $q_{ni} = (\mathcal A f)_{ni}$. Then $\{(q_{n0} - q_{n1}) (x)\}_{n \ge 1} \in l_1$ and $\|(q_{n0} - q_{n1}) (x)\|_{l_1} \le C(\Omega)Z\|f\|_m$.
\end{lem}

\begin{proof}
    Using the definition \eqref{defA} of the operator $\mathcal{A}$ and \eqref{intF1F2}, we obtain
    $$
    |q_{n0} - q_{n1}| \le \Bigg| \int\limits_{0}^{x} F_1(t)( \tilde\varphi_{n0}^{(1)}(t) - \tilde\varphi_{n0}^{(2)}(t) ) \, dt \Bigg| + \Bigg| \int\limits_{0}^{x} F_2(t)( \tilde\varphi_{n0}^{(2)}(t) - \tilde\varphi_{n1}^{(2)}(t) ) \, dt \Bigg|,
    $$
    where $F_1(t)$ and $F_2(t)$ are defined by \eqref{F1_12}-\eqref{F2_12}.

    For the first term we can get:
    $$
    \Bigg| \int\limits_{0}^{x} F_1(t)( \tilde\varphi_{n0}^{(1)}(t) - \tilde\varphi_{n0}^{(2)}(t) ) \, dt \Bigg| \le \delta_n|k_n(x)| + O(\delta_n^2)\Bigg| \int\limits_{0}^{x} F_1(t) \, dt \Bigg|,
    $$
    where $k_n(x) = \Bigg| \int\limits_{0}^{x} F_1(t) t \sin\rho_{n1}t \, dt \Bigg|$. Next, we apply that 
$$    
    \{k_n\} \in l_2 \quad \text{and} \quad \|\{k_n\}\|_{l_2} \le C(\Omega)\|F_1(t)\|_{L_2(0, \pi)} 
$$    
    as the Fourier coefficients. The term with $F_2(t)$ can be investigated in the same way as in Lemma~\ref{lem_qt_12}. Due to these facts, we arrive at the assertion of this lemma.
\end{proof}

Next we can build the specific sequences $\{ g^{(i)}_{nj}(x) \}_{n \ge 1}$ by the formula $g^{(i)}_{nj}(x) = \tilde\varphi^{(i)}_{nj}(x) - \varphi^{(i)}_{nj}(x)$, $n \ge 1$, $j=0, 1$, $i = 1, 2$.

\begin{lem}
    \label{lem_g_12}
    For $S^{(1)}, S^{(2)} \in \mathcal{B}_{\Omega, K}$ we have:
    \begin{enumerate}
        \item $\{g^{(1)}_{n1}(x) - g^{(2)}_{n1}(x)\} \in l_2$ and $\|\{g^{(1)}_{n1}(x) - g^{(2)}_{n1}(x)\}\|_{l_2} \le C(\Omega, K)$,
        \item $\{g^{(1)}_{n0}(x) - g^{(2)}_{n0}(x) - g^{(1)}_{n1}(x) + g^{(2)}_{n1}(x)\} \in l_1$ and $\|\{g^{(1)}_{n0}(x) - g^{(2)}_{n0}(x) - g^{(1)}_{n1}(x) + g^{(2)}_{n1}(x)\}\|_{l_1} \le C(\Omega, K)$,
    \end{enumerate}
    where $n \ge 1$, $x \in [0, \pi]$.
\end{lem}

This lemma can be simply proved by applying Lemmas~\ref{lem_qt_12} and \ref{lem_diff_qt_12} with the sequence $f = \psi(x)$.

Next, for $S^{(i)}$ and $\varphi^{(i)}(x)$, we can build the functions $\sigma^{(i)}(x)$ using \eqref{rec_sigma}.

\begin{lem} \label{lem_sigma_est_12}
For each $S^{(1)}, S^{(2)} \in \mathcal{B}_{\Omega, K}$, we have $\|\sigma^{(1)} - \sigma^{(2)}\|_{L_2(0, \pi)} \le C(\Omega, K)Z$.
\end{lem}

\begin{proof}

    From the reconstruction formula \eqref{rec_sigma} we can get:
    $$
    \sigma^{(1)}(x) - \sigma^{(2)}(x) = \sum\limits_{j=1}^{2}\sum\limits_{i=0}^{1}\sum\limits_{k=1}^{\infty} (-1)^{i+j}\alpha_{ki}^{(j)}\big( 2 \tilde\varphi_{ki}^{(j)}(x)\varphi_{ki}^{(j)}(x) - 1 \big).
    $$
    This sum can be represented as follows:
    $$
    \sigma^{(1)}(x) - \sigma^{(2)}(x) = \sum\limits_{j=1}^{7}S_j(x),
    $$
    where
    \begin{align*}
        &S_1(x) = \sum_{k} (\alpha_{k1}^{(1)} - \alpha_{k1}^{(2)})(2(\tilde\varphi_{k0}^{(1)}(x))^2 - 1), \\
        &S_2(x) = \sum_k \alpha_{k0}^{(2)}(\tilde\varphi_{k0}^{(1)}(x) - \tilde\varphi_{k0}^{(2)}(x))(\tilde\varphi_{k0}^{(1)}(x) + \tilde\varphi_{k0}^{(2)}(x)), \\
        &S_3(x) = \sum_{k} \alpha_{k0}^{(2)}g_{k0}^{(1)}(\tilde\varphi_{k0}^{(1)}(x) - \tilde\varphi_{k0}^{(2)}(x)), \\
        &S_4(x) = \sum_{k} (\alpha_{k1} - \alpha_{k0}^{(2)}) \tilde\varphi_{k1}^{(2)}(x)(g_{k1}^{(1)}(x) - g_{k1}^{(2)}(x)), \\
        &S_5(x) = \sum_{k} \alpha_{k0}^{(2)} (\tilde\varphi_{k1}^{(2)}(x) - \tilde\varphi_{k0}^{(2)}(x))(g_{k1}^{(1)}(x) - g_{k1}^{(2)}(x)), \\
        &S_6(x) = \sum_{k} \alpha_{k0}^{(2)}\tilde\varphi_{k0}^{(2)}(x))(g_{k1}^{(1)}(x) - g_{k1}^{(2)}(x) - g_{k0}^{(1)}(x) + g_{k0}^{(2)}(x)), \\
        &S_7(x) = \sum_{k}(\alpha_{k1} - \alpha_{k0}^{(2)}) \tilde\varphi_{k0}^{(1)}(x)g_{k0}^{(1)}(x).
    \end{align*}

    In view of \eqref{weight_asymp}, \eqref{defBO}, and \eqref{defZ} we have 
    \begin{equation} \label{estal12}
       |\alpha^{(i)}_{kj}| \le C(\Omega), \quad |\alpha_{k1} - \alpha^{(i)}_{k0}| \le \xi_k, \quad |\alpha^{(1)}_{k0} - \alpha^{(2)}_{k0}| \le \delta_k
    \end{equation}
    for any $S^{(i)} \in B_{\Omega}$. Recall that $\{ \xi_n \} \in l_2$ and $\{ \delta_n \} \in l_2$. Consequently, the series $S_1(x)$ converges in $L_2(0,\pi)$ and $\|S_1(x)\|_{L_2(0, \pi)} \le C(\Omega)Z$. The series $S_2(x)$ can be estimated analogously by recalling that $\tilde \varphi^{(i)}_{kj}(x) = \cos (\rho^{(i)}_{kj} x)$ and the asymptotics \eqref{eigen_asymp}.
    Furthermore, using \eqref{estal12} together with estimates 1, 3 of Lemma~\ref{lem_est}, estimate 2 of Lemma~\ref{lem_tilde_est}, estimate 1 of Lemma~\ref{lem_tilde_est_12} and estimates 1--2 of Lemma~\ref{lem_g_12}, we conclude that, for $j = \overline{3,7}$, the series $S_j(x)$ converge absolutely and uniformly on $[0,\pi]$. Furthermore $|S_j(x)| \le C(\Omega, K)Z$. The sum of the resulting estimates yields the claim. 
\end{proof}

\begin{lem}
\label{diff_cdj_12}
    For each $S^{(1)}, S^{(2)} \in \mathcal{B}_{\Omega, K}$ we have $|c_j^{(1)} - c_j^{(2)}| \le C(\Omega, K)Z$, $|d_j^{(1)} - d_j^{(2)}| \le C(\Omega, K)Z$, $j = \overline{0, p}$.
\end{lem}

\begin{proof}

    Here we provide the proof only for $|c_j^{(1)} - c_j^{(2)}|$. For $|d_j^{(1)} - d_j^{(2)}|$, the proof is analogous.

    Consider the compact set
    $$
    \Lambda_{\mathcal{E}, \varepsilon, S^{(1)}, S^{(2)}} = \{ \la \in \mathbb{C}: |\la| \le \mathcal{E}, |\la - \la_{ni}^{(j)}| \ge \varepsilon n^2, \quad n \ge 1, \quad i=0, 1, \quad j=1, 2 \}, \quad \mathcal{E}, \varepsilon > 0.
    $$
    Let us prove the uniform estimate
    \begin{gather}
        \label{diff_r1_12}
        |r_1^{(1)}(\la) - r_1^{(2)}(\la)| \le CZ, \quad S^{(1)}, S^{(2)} \in \mathcal{B}_{\Omega, K}, \quad \la \in \Lambda_{\mathcal{E}, \varepsilon, S^{(1)}, S^{(2)}}.
    \end{gather}
    Here and below, we denote by $C$ various positive constants depending on $\Omega$, $K$, $\mathcal{E}$, $\varepsilon$.
    
    From the reconstruction formula \eqref{rec_r1}, we can get:
    $$
    r_1^{(1)}(\la) - r_1^{(2)}(\la) = \sum\limits_{j=1}^{2} (-1)^{j+1} \prod\limits_{k=1}^{p} (\la - \la_{k0}^{(j)}) \prod\limits_{k=p+1}^{\infty}\dfrac{\la-\la_{k0}^{(j)}}{\la-\la_{k1}^{(j)}}\bigg( 1 - \sum\limits_{k=1}^{\infty} \dfrac{\alpha_{k0}^{(j)}\tilde\varphi_{k0}^{[1] (j)}(\pi)\varphi_{k0}^{(j)}(\pi)}{\la-\la_{k0}^{(j)}} \bigg).
    $$

    This expression can be represented in the following form:
    \begin{gather*}
        r_1^{(1)}(\la) - r_1^{(2)}(\la) = J_1(\la) + J_2(\la) + J_3(\la),
    \end{gather*}
    where
    \begin{align*}
        & J_1(\la) = \bigg( \prod\limits_{k=1}^p \Big( \la - \la^{(1)}_{k0} \Big) - \prod\limits_{k=1}^p \Big( \la - \la^{(2)}_{k0} \Big) \bigg) \prod\limits_{k=p+1}^{\infty}\dfrac{\la-\la_{k0}^{(1)}}{\la-\la_{k1}^{(1)}}\bigg( 1 - \sum\limits_{k=1}^{\infty} \dfrac{\alpha_{k0}^{(1)}\tilde\varphi_{k0}^{[1] (1)}(\pi)\varphi_{k0}^{(1)}(\pi)}{\la-\la_{k0}^{(1)}} \bigg),\\
        & J_2(\la) = \prod\limits_{k=1}^p \Big( \la - \la^{(2)}_{k0} \Big) \bigg( \prod\limits_{k=p+1}^{\infty}\dfrac{\la-\la_{k0}^{(1)}}{\la-\la_{k1}^{(1)}} - \prod\limits_{k=p+1}^{\infty}\dfrac{\la-\la_{k0}^{(2)}}{\la-\la_{k1}^{(2)}} \bigg)\bigg( 1 - \sum\limits_{k=1}^{\infty} \dfrac{\alpha_{k0}^{(1)}\tilde\varphi_{k0}^{[1] (1)}(\pi)\varphi_{k0}^{(1)}(\pi)}{\la-\la_{k0}^{(1)}} \bigg),\\
        & J_3(\la) = \prod\limits_{k=1}^p \Big( \la - \la^{(2)}_{k0} \Big)\prod\limits_{k=p+1}^{\infty}\dfrac{\la-\la_{k0}^{(2)}}{\la-\la_{k1}^{(2)}}\bigg( \sum\limits_{k=1}^{\infty} \dfrac{\alpha_{k0}^{(2)}\tilde\varphi_{k0}^{[1] (2)}(\pi)\varphi_{k0}^{(2)}(\pi)}{\la-\la_{k0}^{(2)}} - \sum\limits_{k=1}^{\infty} \dfrac{\alpha_{k0}^{(1)}\tilde\varphi_{k0}^{[1] (1)}(\pi)\varphi_{k0}^{(1)}(\pi)}{\la-\la_{k0}^{(1)}} \bigg).
    \end{align*}

    Using Lemma~\ref{lem_cj_diff}, we get:
    \begin{align}
        \label{J1}
        & |J_1(\la)| \le C\Bigg| \prod\limits_{k=1}^p \Big( \la - \la^{(1)}_{k0} \Big) - \prod\limits_{k=1}^p \Big( \la - \la^{(2)}_{k0} \Big) \Bigg|,\\
        \label{J2}
        & |J_2(\la)| \le C \Bigg| \prod\limits_{k=p+1}^{\infty}\dfrac{\la-\la_{k0}^{(1)}}{\la-\la_{k1}^{(1)}} - \prod\limits_{k=p+1}^{\infty}\dfrac{\la-\la_{k0}^{(2)}}{\la-\la_{k1}^{(2)}} \Bigg|,\\
        \label{J3}
        & |J_3(\la)| \le C \Bigg| \sum\limits_{k=1}^{\infty} \dfrac{\alpha_{k0}^{(2)}\tilde\varphi_{k0}^{[1] (2)}(\pi)\varphi_{k0}^{(2)}(\pi)}{\la-\la_{k0}^{(2)}} - \sum\limits_{k=1}^{\infty} \dfrac{\alpha_{k0}^{(1)}\tilde\varphi_{k0}^{[1] (1)}(\pi)\varphi_{k0}^{(1)}(\pi)}{\la-\la_{k0}^{(1)}} \Bigg|.
    \end{align}

    Let us start with the finite product from \eqref{J1}. To get the estimate for this expression, we will use mathematical induction. So, we need to prove that
    \begin{gather}
        \label{finite_est_12}
        \Bigg| \prod\limits_{k=1}^p \Big( \la - \la^{(1)}_{k0} \Big) - \prod\limits_{k=1}^p \Big( \la - \la^{(2)}_{k0} \Big)\Bigg| \le CZ.    
    \end{gather}

    Suppose that $p=1$. Then, in the view of \eqref{defZ}, there holds
    $$
    |\la - \la_{10}^{(1)} - \la + \la_{10}^{(2)}| = |\rho_{10}^{(2)} - \rho_{10}^{(1)}||\rho_{10}^{(2)} + \rho_{10}^{(1)}| \le CZ.
    $$
    
    Suppose that, for any $p = \overline{1, n-1}$, the inequality \eqref{finite_est_12} holds. Then we get:
    \begin{align*}
        &\Bigg| \prod\limits_{k=1}^n \Big( \la - \la^{(1)}_{k0} \Big) - \prod\limits_{k=1}^n \Big( \la - \la^{(2)}_{k0} \Big)\Bigg| \le  |\la|\Bigg| \prod\limits_{k=1}^{n-1} \Big( \la - \la^{(1)}_{k0} \Big) - \prod\limits_{k=1}^{n-1} \Big( \la - \la^{(2)}_{k0} \Big)\Bigg| + \\
        & |\la_{n0}^{(2)}|\Bigg| \prod\limits_{k=1}^{n-1} \Big( \la - \la^{(1)}_{k0} \Big) - \prod\limits_{k=1}^{n-1} \Big( \la - \la^{(2)}_{k0} \Big)\Bigg| + |\la_{n0}^{(2)} - \la_{n0}^{(1)}| \Bigg| \prod\limits_{k=1}^{n-1} \Big( \la - \la^{(1)}_{k0} \Big)\Bigg|\le CZ.
    \end{align*}

    Next, let us introduce the difference of the infinite products from \eqref{J2}.
    \begin{align*}
        \Bigg| \prod\limits_{k=p+1}^{\infty}\dfrac{\la-\la_{k0}^{(1)}}{\la-\la_{k1}^{(1)}} - \prod\limits_{k=p+1}^{\infty}\dfrac{\la-\la_{k0}^{(2)}}{\la-\la_{k1}^{(2)}} \Bigg| \le & \Bigg| \prod\limits_{k=p+1}^{\infty}\dfrac{1-\frac{\la}{\la_{k0}^{(1)}}}{1-\frac{\la}{\la_{k1}}}\Bigg|\Bigg| \prod\limits_{k=p+1}^{\infty}\dfrac{\la_{k0}^{(1)}}{\la_{k1}} - \prod\limits_{k=p+1}^{\infty}\dfrac{\la_{k0}^{(2)}}{\la_{k1}} \Bigg| + \\
        & \Bigg| \prod\limits_{k=p+1}^{\infty}\dfrac{\la_{k0}^{(2)}}{\la_{k1}} \Bigg| \Bigg| \prod\limits_{k=p+1}^{\infty}\dfrac{1-\frac{\la}{\la_{k0}^{(1)}}}{1-\frac{\la}{\la_{k1}}} - \prod\limits_{k=p+1}^{\infty}\dfrac{1-\frac{\la}{\la_{k0}^{(2)}}}{1-\frac{\la}{\la_{k1}}} \Bigg|.
    \end{align*}

    Now, we investigate each term separately.

    \begin{enumerate}

        \item $\Bigg| \prod\limits_{k=p+1}^{\infty}\dfrac{\la_{k0}^{(2)}}{\la_{k1}} \Bigg|$.
        Let us take the logarithm of this product:
        \begin{align*}
            \ln \prod\limits_{k=p+1}^{\infty}\dfrac{\la_{k0}^{(2)}}{\la_{k1}} = \sum\limits_{k=p+1}^{\infty} \ln\dfrac{\la_{k0}^{(2)}}{\la_{k1}} = \sum\limits_{k=p+1}^{\infty} \Bigg( \dfrac{2\varkappa_{k}^{(2)}}{k^2} + O\Bigg( \dfrac{(\xi_{k}^{(2)})^{2}}{k^2} \Bigg)\Bigg).
        \end{align*}
        Then,
        $$
        \Bigg| \ln \prod\limits_{k=p+1}^{\infty}\dfrac{\la_{k0}^{(2)}}{\la_{k1}} \Bigg| \le C,
        $$
        and, consequently,
        $$
        \Bigg| \prod\limits_{k=p+1}^{\infty}\dfrac{\la_{k0}^{(2)}}{\la_{k1}} \Bigg| \le C.
        $$

        \item $\Bigg| \prod\limits_{k=p+1}^{\infty}\dfrac{\la_{k0}^{(1)}}{\la_{k1}} - \prod\limits_{k=p+1}^{\infty}\dfrac{\la_{k0}^{(2)}}{\la_{k1}} \Bigg|$.

        Here we will subsequently separate the terms with $|\la_{k0}^{(1)} - \la_{k0}^{(2)}|$. So,

        $$
        \Bigg| \prod\limits_{k=p+1}^{\infty}\dfrac{\la_{k0}^{(1)}}{\la_{k1}} - \prod\limits_{k=p+1}^{\infty}\dfrac{\la_{k0}^{(2)}}{\la_{k1}} \Bigg| \le \sum\limits_{k=p+1}^{\infty} \Bigg| \dfrac{\la_{k0}^{(1)} - \la_{k0}^{(2)}}{\la_{k1}} \Bigg| \Bigg| \prod\limits_{n=p+1}^{k-1}\dfrac{\la_{n0}^{(2)}}{\la_{n1}} \Bigg| \Bigg| \prod\limits_{n=k+1}^{\infty}\dfrac{\la_{n0}^{(1)}}{\la_{n1}} \Bigg|.
        $$

        Due to the previous point and taking the estimate
        $$
        \Bigg| \dfrac{\la_{k0}^{(1)} - \la_{k0}^{(2)}}{\la_{k1}} \Bigg| \le \Bigg| \dfrac{2\delta_k}{k} \Bigg| + \Bigg| \dfrac{\delta_k\varkappa_k^{(1)}}{k^2} \Bigg| + \Bigg| \dfrac{\delta_k\varkappa_k^{(1)}}{k^2} \Bigg|,
        $$
        into account, we obtain
        $$
        \Bigg| \prod\limits_{k=p+1}^{\infty}\dfrac{\la_{k0}^{(1)}}{\la_{k1}} - \prod\limits_{k=p+1}^{\infty}\dfrac{\la_{k0}^{(2)}}{\la_{k1}} \Bigg| \le CZ.
        $$
        
        \item $\Bigg| \prod\limits_{k=p+1}^{\infty}\dfrac{1-\frac{\la}{\la_{k0}^{(1)}}}{1-\frac{\la}{\la_{k1}}}\Bigg| \le C$.

        This point can be proved with the same technique, as in the first point. 
        
        \item $\Bigg| \prod\limits_{k=p+1}^{\infty}\dfrac{1-\frac{\la}{\la_{k0}^{(1)}}}{1-\frac{\la}{\la_{k1}}} - \prod\limits_{k=p+1}^{\infty}\dfrac{1-\frac{\la}{\la_{k0}^{(2)}}}{1-\frac{\la}{\la_{k1}}} \Bigg| \le CZ$.

        This point can be proved with the same technique, as in the second point. 
    \end{enumerate}

    Then, we conclude that
    \begin{gather}
    \label{infinite_r1_12}
    \Bigg| \prod\limits_{k=p+1}^{\infty}\dfrac{\la-\la_{k0}^{(1)}}{\la-\la_{k1}^{(1)}} - \prod\limits_{k=p+1}^{\infty}\dfrac{\la-\la_{k0}^{(2)}}{\la-\la_{k1}^{(2)}} \Bigg| \le CZ.
    \end{gather}

    Finally, using the arguments from the proof of Lemma~\ref{lem_sigma_est_12}, we get the following estimate:
    \begin{gather}
    \label{sums_r1_12}
    \Bigg| \sum\limits_{k=1}^{\infty} \dfrac{\alpha_{k0}^{(2)}\tilde\varphi_{k0}^{[1] (2)}(\pi)\varphi_{k0}^{(2)}(\pi)}{\la-\la_{k0}^{(2)}} - \sum\limits_{k=1}^{\infty} \dfrac{\alpha_{k0}^{(1)}\tilde\varphi_{k0}^{[1] (1)}(\pi)\varphi_{k0}^{(1)}(\pi)}{\la-\la_{k0}^{(1)}} \Bigg| \le CZ.
    \end{gather}

    So, substituting \eqref{finite_est_12}, \eqref{infinite_r1_12}, \eqref{sums_r1_12} into \eqref{J1}, \eqref{J2}, \eqref{J3} respectively, we get that $|J_j(\la)| \le CZ$, $j=\overline{1, 3}$, $\la \in \Lambda_{\mathcal{E}, \varepsilon, S^{(1)}, S^{(2)}}$, and consequently the estimate \eqref{diff_r1_12}.

    Using the technique from step~2 of the proof of Lemma~\ref{lem_cj_diff}, we obtain the desired estimates for the polynomial coefficients.
\end{proof}

Lemmas~\ref{reconstruction}, \ref{lem_sigma_est_12}, and \ref{diff_cdj_12} together imply Theorem~\ref{thm_uni_stab}.

\section{Non-solvability conditions} \label{sec:solvability}

Let us discuss the case when the spectral data $S = \{ \la_n, \alpha_n\}_{n \ge 1}$ contain multiple eigenvalues or zero weight numbers. In principle, Theorems~\ref{thm_uni_bound} and~\ref{thm_uni_stab} are valid for this case. 
However, the invertibility for the operator $(E + \tilde H(x))$ is crucial for the inverse problem solvability. In this section, we show that this operator is not invertible if eigenvalue multiplicities or the number of zero weight numbers are too large.
 
Suppose that $S \in \mathcal S_p$, $p \ge 0$. Denote 
\begin{align*}
& A = A(S) := \# \{ n \ge 1 \colon \alpha_n =  0\}, \\
& \mathcal I = \mathcal I(S) := \{ n \ge 1 \colon (n = 1 \: \text{or} \: \la_n \ne \la_k \: \text{for all} \: k < n) \: \text{and} \: \alpha_n \ne 0\}, \\
& m_k = m_k(S) := \# \{ n \ge 1 \colon \la_n = \la_k \: \text{and} \: \alpha_n \ne 0\}, \quad k \in \mathcal I.
\end{align*}
 
Thus, $\mathcal I$ is the index set of all the distinct eigenvalues among $\{ \la_n \}_{n \ge 1}$ excluding the ones corresponding to $\alpha_n = 0$, and $m_k$ is the multiplicity of $\la_k$. In view of \eqref{eigen_asymp}, we have $m_k = 1$ for all sufficiently large indices $k$, so the following sum is finite:
$$
B = B(S) := \sum_{k \in \mathcal I} (m_k - 1).
$$

\begin{thm} \label{thm:inv}
Suppose that $S \in \mathcal S_p$ and $A(S) + B(S) \ge p+1$.
Then, the operator $E+\tilde H(\pi)$ is non-invertible on $m$.
\end{thm}

In order to prove Theorem~\ref{thm:inv}, we need the following auxiliary fact, which is similar to \cite[Lemma~4.1]{BondR25}:

\begin{lem}
\label{lem:solv}
    Let $x \in [0, \pi]$ be fixed. Then for $S \in \mathcal{S}_p$ the operator $E+\tilde H(x)$ is invertible on $m$ iff the following infinite system
    \begin{align}
        \notag
        \gamma_{ni} & + \sum\limits_{k=1}^{p+1}\alpha_{k0}\tilde D(x, \la_{ni}, \la_{k0})\gamma_{k0} -\dfrac{1}{\pi}\tilde D(x, \la_{ni}, 0)\gamma_{11} \\
        \label{system_inv}
        &+ \sum\limits_{k=p+2}^{\infty}(\alpha_{k0}\tilde D(x, \la_{ni}, \la_{k0})\gamma_{k0} - \alpha_{k1}\tilde D(x, \la_{ni}, \la_{k1})\gamma_{k1}) = 0, \quad n \ge 1, \quad i=0,1
    \end{align}
    does not have a non-trivial solution $\{ \gamma_{ni} \}$ in $m$, satisfying the estimates
    \begin{gather}
    \label{est_sol}
        |\gamma_{ni}| \le 1, \quad |\gamma_{n0} - \gamma_{n1}| \le |\hat\rho_n|, \quad n \ge 1, \quad i = 0, 1.
    \end{gather}
\end{lem}

\begin{proof}[Proof of Theorem~\ref{thm:inv}]
    At first, consider equation~\eqref{system_inv} at the point $x=\pi$:
    \begin{align}
        \notag
        \gamma_{ni} & + \sum\limits_{k=1}^{p+1}\alpha_{k0}\tilde D(\pi, \la_{ni}, \la_{k0})\gamma_{k0} -\dfrac{1}{\pi}\tilde D(\pi, \la_{ni}, 0)\gamma_{11} \\
        \label{system_inv_pi}
        &+ \sum\limits_{k=p+2}^{\infty}(\alpha_{k0}\tilde D(\pi, \la_{ni}, \la_{k0})\gamma_{k0} - \alpha_{k1}\tilde D(\pi, \la_{ni}, \la_{k1})\gamma_{k1}) = 0, \quad n \ge 1, \quad i=0,1.
    \end{align}
    According to \eqref{D_def} and \eqref{model_sd}, we have:
    \begin{gather}
    \label{D_at_pi}
        \tilde D(\pi, \la_{n1}, \la_{k1}) = \left\{ \begin{aligned} 
            \pi, \quad n, k \le p + 1,\\
            \frac{\pi}{2} \delta_{nk}, \quad n , k > p + 1, \\
            0, \quad \text{otherwise},
            \end{aligned} \right. 
    \end{gather}
    where $\delta_{nk}$ is the Kronecker delta. Substituting \eqref{D_at_pi} into \eqref{system_inv_pi}, we get for $i=1$:
    \begin{gather}
        \label{system_i}
        \left\{ \begin{aligned} 
            \gamma_{n1}+\sum_k \alpha_{k0}\tilde D(\pi, 0, \la_{k0})\gamma_{k0}-\gamma_{11}=0, \quad n \le p + 1,\\
            \sum_k \alpha_{k0}\tilde D(\pi, \la_{n1}, \la_{k0})\gamma_{k0} = 0, \quad n  \ge p + 2.
            \end{aligned} \right. 
    \end{gather}
    
    In order to show the main idea, we consider two cases $A(S) = 0$ and $B(S) = 0$. The general case can be obtained by their combination.
    
    \textbf{\textit{Case 1:}} Let all $\alpha_k \neq 0$, $k \ge 1$, and $B(S) \ge p + 1$.
    Let $\gamma_{n0} = 0$, $n \ge 1$. Then, from \eqref{system_i} we get that $\gamma_{n1}=\gamma_{11}$, $2 \le n \le p+ 1$. Next, for $i=0$, the system \eqref{system_inv_pi} takes the form
    $$
    \alpha_{11}\gamma_{11}\tilde D(\pi, \la_{n0}, 0) + \sum\limits_{k=p+2}^\infty \alpha_{k1}\gamma_{k1}\tilde D(\pi, \la_{n0}, \la_{k1}) = 0, \quad n \ge 1.
    $$
    Due to \eqref{D_def} we get:
    \begin{gather}
    \label{int_eq}
    \int\limits_{0}^{\pi}\gamma(t)\cos\rho_{n0}t\,dt = 0, \quad n \ge 1, \quad \gamma(t) = \sum\limits_{k=p+2}^\infty \frac{2}{\pi}\gamma_{k1}\cos\rho_{k1}t + \frac{1}{\pi}\gamma_{11}.
    \end{gather}
    As $B(S) \ge p + 1$ and the asymptotics \eqref{eigen_asymp} holds for $\rho_{n0}$, then $\{ \cos\rho_{n0}t \}_{n\ge 1}$ is not complete in $L_2(0, \pi)$. Then, there exists a non-trivial function $\gamma(t)$ satisfying \eqref{int_eq}. It means, that $\gamma_{11}$ (and consequently $\gamma_{n1}$, $2 \le n \le p+1$) and $\gamma_{k1}$, $k \ge p+2$, can be found from the Fourier coefficients of $\gamma(t)$ with respect to the orthogonal basis $\{ \cos(k-p-1)t \}_{k \ge p+1}$. So, we have found a non-trivial solution of the system \eqref{system_inv_pi}. Moreover, using \eqref{int_eq} we can simply get the estimates \eqref{est_sol} as follows:
    \begin{align*}
        |\gamma_{n0}-\gamma_{n1}| & = \Bigg| \int\limits_0^\pi \gamma(t)\cos\rho_{n0}t \, dt - \int\limits_0^\pi \gamma(t)\cos\rho_{n1}t \, dt \Bigg| \\
        &\le \| \gamma(t) \|_{L_2(0, \pi)} \sqrt{\int\limits_0^\pi (|\hat\rho_n^2|t^2\sin^2(n-p-1)t + O(|\hat\rho_n^3|)) \, dt} \\
        & \le C|\hat\rho_n|.
    \end{align*}
    Dividing the solution by the constant $C$, we arrive at \eqref{est_sol}.
    Then, by virtue of Lemma~\ref{lem:solv}, the operator $E+\tilde H(\pi)$ is non-invertible on $m$.    
    
    \textbf{\textit{Case 2:}} Let all $\la_{n0}$ be simple. Without loss of generality suppose that the first $p+1$ values $\alpha_{n} := \alpha_{n0}$ are zeros: $\alpha_1 = \alpha_2 = \dots = \alpha_{p+1} = 0$.
    Let $\gamma_{n0} = 0$, $n \ge p+2$, and $\gamma_{p+1, 0} = 1$. The values $\gamma_{n0}$, $1 \le 1 \le p$, are considered unknown. Then, from \eqref{system_i} we get that $\gamma_{n1} = \gamma_{11}$, $2 \le n \le p+1$. Next, for $i=0$ system \eqref{system_inv_pi} takes the form
    $$
    \sum\limits_{k=p+2}^\infty \frac{2}{\pi}\tilde D(\pi, \la_{n0}, \la_{k1})\gamma_{k1} + \frac{1}{\pi}\tilde D(\pi, \la_{n0}, 0)\gamma_{11}= \gamma_{n0}, \quad n \ge 1.
    $$
    Due to \eqref{D_def}, we get:
    \begin{gather}
        \label{int_case2_1}
        \int\limits_{0}^{\pi}\gamma(t)\cos\rho_{n0}t\,dt = \gamma_{n0}, \quad 1 \le n \le p, \\
        \label{int_case2_2}
        \int\limits_{0}^{\pi}\gamma(t)\cos\rho_{p+1, 0}t\,dt = 1, \\
        \label{int_case2_3}
        \int\limits_{0}^{\pi}\gamma(t)\cos\rho_{n0}t\,dt = 0, \quad n \ge p+2,\\
        \notag
        \gamma(t) = \frac{1}{\pi}\gamma_{11} + \sum\limits_{k=p+2}^\infty \frac{2}{\pi}\gamma_{k1}\cos\rho_{k1}t.
    \end{gather}
    
    As all $\rho_{n0}$ are distinct and due to \eqref{eigen_asymp} we conclude that $\{ \cos\rho_{n0}t \}_{n\ge p+1}$ is the Riesz basis in $L_2(0, \pi)$. Then, there exists the function $\gamma(t) \in L_2(0, \pi)$, which is uniquely defined by its Fourier coefficients with respect to the Riesz basis $\{ \cos\rho_{n0}t \}_{n \ge p+1}$, satisfying the system \eqref{int_case2_2}--\eqref{int_case2_3}. Using $\gamma(t)$, we can find the unknown values $\gamma_{n0}$, $1 \le n \le p$ by \eqref{int_case2_1}. Finally, we can obtain all $\gamma_{n1}$, $n \ge 1$, as the Fourier coefficients of $\gamma(t)$ with respect to the orthogonal basis $\{ \cos(k-p-1)t \}_{k \ge p+1}$. Moreover, using \eqref{int_case2_1}, we can simply get the estimates \eqref{est_sol}. So, we have found a non-trivial solution of the system \eqref{system_inv_pi}. Then, according to Lemma~\ref{lem:solv}, the operator $E+\tilde H(\pi)$ is non-invertible on $m$.
\end{proof}

\section{Examples} \label{sec:ex}

In this section, we present two examples, which demostrate the stability of the inverse problem for different degrees of the polynomials in the boundary conditions \eqref{bc1} for two problems $L^{(1)}$ and $L^{(2)}$.

Let us suppose that $p=1$. Then, the spectral data of the model problem $\tilde L = L(0, \la, 0)$ have the following form:
$$
\la_{n1} = \left\{ \begin{aligned} 
    0, \quad n=1,2, \\
    (n-2)^2, \quad n \ge 3; \\
\end{aligned} \right. \qquad
\alpha_{n1} = \left\{ \begin{aligned} 
    \dfrac{1}{\pi}, \quad n=1, \\
	0, \quad n=2, \\
    \dfrac{2}{\pi}, \quad n \ge 3. \\
\end{aligned} \right.
$$

\begin{example} \label{ex:1}
Introduce the spectral data of the problems $L^{(1)}$ and $L^{(2)}$ as follows:
$$
\la_{n0}^{(1)} = \left\{ \begin{aligned}
0, \quad n=1, \\
\eta^2, \quad n=2, \\
(n-2)^2, \quad n \ge 3;
\end{aligned} \right. \qquad
\alpha_{n0}^{(1)} = \left\{ \begin{aligned} 
\dfrac{1}{\pi}, \quad n=1, \\
0, \quad n=2, \\
\dfrac{2}{\pi}, \quad n \ge 3;
\end{aligned} \right.
$$
$$
\la_{n0}^{(2)} = \left\{ \begin{aligned} 
0, \quad n=1, \\
\eta^2, \quad n=2, \\
(n-2)^2, \quad n \ge 3;
\end{aligned} \right. \qquad
\alpha_{n0}^{(2)} = \left\{ \begin{aligned} 
\dfrac{1}{\pi}, \quad n=1, \\
\alpha, \quad n=2, \\
\dfrac{2}{\pi}, \quad n \ge 3,
\end{aligned} \right.
$$
where $\alpha \neq 0$ is a complex number, small by absolute value, and $\eta \neq 0$ is some complex number.

The corresponding Weyl functions have the form
\begin{align*}
M^{(1)}(\la) & = \frac{1}{\pi \la} + \frac{0}{\la -\eta^2}  + \sum_{n = 3}^{\infty} \frac{\alpha_{n1}}{\la - \la_{n1}}, \\
M^{(2)}(\la) & = \frac{1}{\pi \la} + \frac{\alpha}{\la -\eta^2}  + \sum_{n = 3}^{\infty} \frac{\alpha_{n1}}{\la - \la_{n1}}. \\
\end{align*}

In view of~\eqref{seriesMt}, there holds
$M^{(1)}(\la) \equiv \tilde M(\la)$. We just add a dummy eigenvalue $\eta^2$ with the zero weight number. The function $M^{(2)}(\la)$ is obtained by a small perturbation of $M^{(1)}(\la)$, which makes this additional weight number to be non-zero and actually adds the eigenvalue $\eta^2$.

Construct the operator $\tilde H^{(1)}(x)$. Using formula \eqref{Hnk}, we get
$$
\tilde H_{nk}^{(1)}(x) = 
\begin{pmatrix}
0 & 0 \\
0 & 0
\end{pmatrix}
, \quad n,k \ge 1.
$$

It means that the main equation \eqref{main_eq} has the form $\tilde\psi(x) = \psi(x)$. Then, $\varphi_{ni}^{(1)}(x) = \tilde \varphi_{ni}^{(1)}(x)$, $n \ge 1$, $i=0,1$. Then, from the reconstruction formulas \eqref{rec_sigma}, \eqref{rec_r1}, and \eqref{rec_r2} we get:
\begin{equation} \label{si0}
\sigma^{(1)}(x) = 0, \quad r_1^{(1)}(\la) = c_0^{(1)} + \la, \quad r_2^{(1)}(\la) = d_0^{(1)} + d_1^{(1)} \la = 0.
\end{equation}

Thus, $L^{(1)} = \tilde L$ as expected.
Next, we build the components of the operator $\tilde H^{(2)}(x)$ by~\eqref{Hnk}:
$$
\tilde H_{nk}^{(2)}(x) = \left\{ \begin{aligned} 
\begin{pmatrix}
0 & 0 \\
0 & 0
\end{pmatrix}
, \quad n \neq 2 \quad k \neq 2, \\
\begin{pmatrix}
\alpha \bigg( \dfrac{x}{2} + \dfrac{\sin2\eta x}{4\eta} + \dfrac{\sin\eta x}{\eta} \bigg) &  \dfrac{\alpha}{\eta} \bigg( \dfrac{x}{2} + \dfrac{\sin2\eta x}{4\eta} + \dfrac{\sin\eta x}{\eta} \bigg) \\
\alpha\sin\eta x & \dfrac{\alpha\sin\eta x}{\eta}
\end{pmatrix}
, \quad n = 2, \quad k = 2    .
\end{aligned} \right.
$$

Then we get $\| \tilde H^{(1)}(x) - \tilde H^{(2)}(x) \|_{m \to m} \le C|\alpha|$ for all $x \in [0,\pi]$. Since the operator $(E + \tilde H^{(1)}(x))$ is obviously invertible for each fixed $x \in [0,\pi]$, so does $(E + \tilde H^{(2)}(x))$ for sufficiently small $|\alpha|$. Moreover, the estimate $\| (E + \tilde H(x))^{-1} \|_{m \to m} \le C$ holds. So, we can solve the main equation \eqref{main_eq} for the spectral data of $L^{(2)}$ and find the functions $\sigma^{(2)}(x)$, 
$$
r_1^{(2)} (\la) = c_0^{(2)} + \la \quad \text{and} \quad r_2^{(2)} (\la) = d_0^{(2)} + d_1^{(2)}\la, 
$$
using \eqref{rec_sigma}, \eqref{rec_r1}, and \eqref{rec_r2} respectively.
Consequently, Theorem~\ref{thm_uni_stab} implies
$$
\|\sigma^{(1)}(x) - \sigma^{(2)}(x)\|_{L_2(0, \pi)} + |c_0^{(1)} - c_0^{(2)}| + |d_0^{(1)} - d_0^{(2)}| + |d_1^{(1)} - d_1^{(2)}| \le C|\alpha|.
$$

On the other hand, the polynomials $r_1^{(1)}(\la)$ and $r_2^{(1)}(\la)$ in \eqref{si0} are not relatively prime, so their degrees can be reduced: $r_1^{(1)}(\la) = 1$, $r_2^{(2)}(\la) = 0$ ($p= 0$). Indeed, the Weyl function $M^{(1)}(\la)$ coincides with the Weyl function for the problem $L(0,1,0)$ with the Neumann-Neumann boundary conditions. At the same time, the polynomials $r_1^{(2)}(\la)$ and $r_2^{(2)}(\la)$ are relatively prime and have degree $1$. Thus, the estimate for $\| \sigma^{(1)} - \sigma^{(2)} \|_{L_2(0,\pi)}$ is actually obtained for the two boundary value problems of form \eqref{eqv1}--\eqref{bc1} with polynomials of different degrees.
\end{example}

\begin{example}
Introduce the spectral data of the problems $L^{(1)}$ and $L^{(2)}$ as follows:
$$
\la_{n0}^{(1)} = \left\{ \begin{aligned} 
    0, \quad n=1,2, \\
    (n-2)^2, \quad n \ge 3;
\end{aligned} \right. \qquad
\alpha_{n0}^{(1)} = \left\{ \begin{aligned} 
    \alpha, \quad n=1, \\
    \dfrac{1}{\pi} - \alpha, \quad n=2, \\
    \dfrac{2}{\pi}, \quad n \ge 3;
\end{aligned} \right.
$$
$$
\la_{n0}^{(2)} = \left\{ \begin{aligned} 
    \la_{n0}^{(2)}, \quad n=1,2, \\
    (n-2)^2, \quad n \ge 3; \\
\end{aligned} \right. \qquad
\alpha_{n0}^{(2)} = \left\{ \begin{aligned} 
    \alpha, \quad n=1, \\
    \dfrac{1}{\pi} - \alpha, \quad n=2, \\
    \dfrac{2}{\pi}, \quad n \ge 3,
\end{aligned} \right.
$$
where $\alpha \neq 0$ is some complex number, $\la_{10}^{(2)} \neq \la_{20}^{(2)}$ are non-zero numbers close to $\la_{11}= 0$.

The corresponding Weyl functions have the form
\begin{align*}
M^{(1)}(\la) & = \frac{\alpha}{\la} + \left( \frac{1}{\pi} - \alpha\right)\frac{1}{\la} + \sum_{n = 3}^{\infty} \frac{\alpha_{n1}}{\la-\la_{n1}}, \\
M^{(2)}(\la) & = \frac{\alpha}{\la -\la_{10}^{(1)}} + \left( \frac{1}{\pi} - \alpha\right)\frac{1}{\la -\la_{20}^{(1)}} + \sum_{n = 3}^{\infty} \frac{\alpha_{n1}}{\la-\la_{n1}}.
\end{align*}

In view of \eqref{seriesMt}, we have $M^{(1)}(\la) \equiv \tilde M(\la)$, but the weight number $\alpha_{11} = \frac{1}{\pi}$ splits between two copies of the zero eigenvalue. Under a small perturbation, these two copies become two different eigenvalues $\la_{10}^{(2)}$ and $\la_{20}^{(2)}$ of $L^{(2)}$.

Let us start from the building of the operator $\tilde H^{(1)}(x)$. Using \eqref{Hnk} and simple calculations, we get
$$
\tilde H_{nk}^{(1)}(x) = \left\{ \begin{aligned} 
    \begin{pmatrix}
        0 & \bigg(\dfrac{1}{\pi}-\alpha \bigg) \rho_{n1}\int\limits_0^{\pi}\sin\rho_{n1}t \, dt \\
        0 & -\bigg(\dfrac{1}{\pi}-\alpha \bigg)\int\limits_0^{\pi}\cos\rho_{n1}t \, dt
    \end{pmatrix}
    , \quad n \ge 1, \quad k = 1, \\
    \begin{pmatrix}
        0 & \alpha\rho_{n1}\int\limits_0^{\pi}\sin\rho_{n1}t \, dt \\
        0 & -\alpha\int\limits_0^{\pi}\cos\rho_{n1}t \, dt
    \end{pmatrix}
    , \quad n \ge 1, \quad k = 2, \\
    \begin{pmatrix}
        0 & 0 \\
        0 & 0
    \end{pmatrix}
    , \quad n \ge 1, \quad k \ge 3.
\end{aligned} \right.
$$

Note that, in this case, it is more convenient to find the components $\varphi_{ni}^{(1)}(x)$ using not the main equation, but the relation of Lemma~\ref{prop_varphi}. So, for $i=0$ and $n=1,2$, we have the following system:
\begin{gather*}
    \left\{ \begin{aligned} 
    \varphi_{10}^{(1)}(x) = 1+x\Bigg( \dfrac{1}{\pi}\varphi(x, 0) - \alpha \varphi_{10}^{(1)}(x) - \bigg( \dfrac{1}{\pi} - \alpha \bigg) \varphi_{20}^{(1)}(x) \Bigg), \\
    \varphi_{20}^{(1)}(x) = 1+x\Bigg( \dfrac{1}{\pi}\varphi(x, 0) - \alpha \varphi_{10}^{(1)}(x) - \bigg( \dfrac{1}{\pi} - \alpha \bigg) \varphi_{20}^{(1)}(x) \Bigg), \\
    \varphi(x, 0) = 1+x\Bigg( \dfrac{1}{\pi}\varphi(x, 0) - \alpha \varphi_{10}^{(1)}(x) - \bigg( \dfrac{1}{\pi} - \alpha \bigg) \varphi_{20}^{(1)}(x) \Bigg).
\end{aligned} \right.
\end{gather*}

Then, $\varphi_{10}^{(1)}(x) = \varphi_{20}^{(1)}(x) = \varphi(x, 0) = 1$. For $i=1$, $n=1,2$ we can analogously get $\varphi_{11}^{(1)}(x) = \varphi_{21}^{(1)}(x) = 1$ and for $i=0,1$, $n \ge 3$: $\varphi_{ni}^{(1)}(x) = \cos(n-2)x$. So, using the reconstruction formulas \eqref{rec_sigma}--\eqref{rec_r2} we arrive at the coefficients \eqref{si0} of the problem $L^{(1)}$. As in Example~\ref{ex:1}, the problem $L^{(1)}$ is equivalent to the problem $L(0,1,0)$ with the Neumann-Neumann boundary conditions, which correspond to the degree $p = 0$ of the polynomials.

One can similarly construct the operator $\tilde H^{(2)}(x)$ and show that 
$$
\| \tilde H^{(1)}(x) - \tilde H^{(2)}(x) \|_{m \to m} \le C\varepsilon, \quad x \in [0,\pi],  \quad \text{if}\:\: |\rho_{n0}^{(2)}| \le \varepsilon, \: n = 1,2,
$$
where $\varepsilon > 0$ is sufficiently small.
Then, we can apply the Theorem~\ref{thm_uni_stab} and obtain the stability estimate 
$$
\|\sigma^{(1)}(x) - \sigma^{(2)}(x)\|_{L_2(0, \pi)} + |c_0^{(1)} - c_0^{(2)}| + |d_0^{(1)} - d_0^{(2)}| + |d_1^{(1)} - d_1^{(2)}| \le C\varepsilon
$$
for the problems $L^{(1)}$ and $L^{(2)}$, which after simplifying have different polynomial degrees.
\end{example}

\medskip

\textbf{Funding}: This work was supported by Grant 24-71-10003 of the Russian Science Foundation, https://rscf.ru/en/project/24-71-10003/.





\medskip

\noindent Natalia Pavlovna Bondarenko \\
1. Department of Applied Mathematics, Samara National Research University,\\
Moskovskoye Shosse 34, Samara 443086, Russia.\\
2. S.M. Nikolskii Mathematical Institute, Peoples' Friendship University of Russia (RUDN University), 6 Miklukho-Maklaya Street, Moscow, 117198, Russia.\\
3. Moscow Center of Fundamental and Applied Mathematics, Lomonosov Moscow State University, Moscow 119991, Russia.\\
e-mail: {\it bondarenkonp@sgu.ru}\\

\medskip

\noindent Egor Evgenevich Chitorkin \\
1. Institute of IT and Cybernetics, Samara National Research University,\\ 
Moskovskoye Shosse 34, Samara 443086, Russia. \\
2. Department of Mechanics and Mathematics, Saratov State University, \\
Astrakhanskaya 83, Saratov 410012, Russia. \\
e-mail: {\it chitorkin.ee@ssau.ru} 

\medskip

\begin{thebibliography}{99}


\bibitem{Lev84}
Levitan, B.M. Inverse Sturm-Liouville Problems, VNU Sci. Press, Utrecht, 1987.

\bibitem{FY01}
Freiling, G.; Yurko, V. Inverse Sturm-Liouville Problems and Their Applications, Nova Science Publishers, Huntington, NY, 2001.

\bibitem{Mar11}
Marchenko, V.A. Sturm-Liouville Operators and Applications. Revised edition, AMS, Providence, 2011.

\bibitem{Krav20}
Kravchenko, V.V. Direct and Inverse Sturm-Liouville Problems, Birkh\"auser, Cham, 2020.


\bibitem{BHM15}
Belinskiy, B.P.; Hiestand, J.W.; Matthews, J.V. Piecewise uniform optimal design of a bar with an attached mass, Electr. J. Differ. Equ. 2015 (2015), no. 206, 1--17

\bibitem{Gran22}
Granet, E. Regularization of energy-dependent pointlike interactions in 1D quantum mechanics, J. Phys. A 55(2022), no. 42, Article ID 425308.

\bibitem{GKR16}
Gwak, S.; Kim, J.; Rey, S.-J. Massless and massive higher spins from anti-de Sitter space waveguide, J. High Energy Phys. 2016 (2016), no.~11, Article ID 024.

\bibitem{Ful99}
Fulton, C.T. Two-point boundary value problems with eigenvalue parameter contained in the
boundary conditions, Proc. Roy. Soc. Edinburgh Sect. A 77 (1977), no. 3–4, 293--308.



\bibitem{Chu01}
Chugunova, M.V. Inverse spectral problem for the Sturm-Liouville operator with eigenvalue parameter dependent boundary conditions. Oper. Theory: Advan. Appl. 123, Birkhauser, Basel (2001), 187--194.

\bibitem{BindBr021}
Binding, P.A.; Browne, P.J.; Watson, B.A. Sturm-Liouville problems with boundary conditions rationally dependent on the eigenparameter. I, Proc. Edinb. Math. Soc. (2) 45 (2002), no. 3, 631--645.

\bibitem{BindBr022}
Binding, P.A.; Browne, P.J.; Watson, B.A. Sturm-Liouville problems with boundary conditions rationally dependent on the eigenparameter. II, J. Comput. Appl. Math. 148 (2002), no. 1, 147--168.

\bibitem{ChFr}
Chernozhukova, A.; Freiling, G. A uniqueness theorem for the boundary value problems with non-linear dependence on the spectral parameter in the boundary conditions, Inv. Probl. Sci. Eng. 17 (2009), no. 6, 777--785.

\bibitem{FrYu}
Freiling, G.; Yurko V. Inverse problems for Sturm-Liouville equations with boundary conditions polynomially dependent on the spectral parameter, Inverse Problems 26 (2010), no. 5, Article ID 055003.

\bibitem{FrYu12}
Freiling, G.; Yurko, V. Determination of singular differential pencils from the Weyl function, Adv. Dyn. Sys. Appl. 7 (2012), no. 2, 171--193.

\bibitem{YangWei18}
Yang, Y.; Wei, G. Inverse scattering problems for Sturm-Liouville operators with spectral parameter dependent on boundary conditions, Math. Notes 103 (2018), no.~1--2, 59--66.

\bibitem{Gul19}
Guliyev, N.J. Schr\"odinger operators with distributional potentials and boundary conditions dependent on the eigenvalue parameter, J. Math. Phys. 60 (2019), Article ID 063501.

\bibitem{Gul20-ams}
Guliyev, N.J. On two-spectra inverse problems, Proc. AMS. 148 (2020), 4491--4502.

\bibitem{Gul23}
Guliyev, N.J. Inverse square singularities and eigenparameter-dependent boundary conditions are two sides of the same coin, 
The Quarterly J. Math. 74 (2023), no. 3, 889--910.


\bibitem{amp} 
Chitorkin, E.E., Bondarenko, N.P. Solving the inverse Sturm-Liouville problem with singular potential and with polynomials in the boundary conditions. Anal. Math. Phys. 13 (2023), Article  ID 79.

\bibitem{LocSolv}
Chitorkin, E.E., Bondarenko, N.P. Local solvability and stability for the inverse Sturm--Liouville problem with polynomials in the boundary conditions, Math. Meth. Appl. Sci. 47 (2024), 8881--8903. 

\bibitem{ChitBondStabMult}
Chitorkin, E.E.; Bondarenko, N.P. Inverse Sturm--Liouville problem with polynomials in the boundary condition and multiple eigenvalues, Journal of Inverse and Ill--Posed Problems (2025), published online. https://doi.org/10.1515/jiip-2024-0011.

\bibitem{Bond25}
Bondarenko, N.P. Uniform stability for the inverse Sturm-Liouville problem with eigenparameter-dependent boundary conditions, arXiv:2502.18289.


\bibitem{SavShkal03}
Savchuk, A.M.; Shkalikov, A.A. Sturm-Liouville operators with distribution potentials, Transl. Moscow Math. Soc. 64 (2003), 143--192.

\bibitem{Sav10}
Savchuk, A.M.; Shkalikov, A.A. Inverse problems for Sturm-Liouville operators with potentials in Sobolev spaces: uniform stability, Funct. Anal. Appl. 44 (2010), no.~4, 270--285.

\bibitem{SS13}
Savchuk, A.M.; Shkalikov, A.A. Uniform stability of the inverse Sturm-Liouville problem with respect to the spectral function in the scale of Sobolev spaces, Proc. Steklov Inst. Math. 283 (2013), 181--196.

\bibitem{SS14}
Savchuk A.M., Shkalikov A.A. Recovering of a potential of the Sturm–Liouville problem from finite sets of spectral data, Spectral Theory and Differential Equations, Amer. Math. Soc. Transl. Ser. 2, 233, Amer. Math. Soc., Providence, RI, 2014, 211--224.

\bibitem{Hryn03}
Hryniv, R.O.; Mykytyuk, Y.V. Inverse spectral problems for Sturm-Liouville operators with singular potentials, Inverse Problems 19 (2003), no.~3, 665--684.

\bibitem{Hryn11}
Hryniv R.O. Analyticity and uniform stability in the inverse singular Sturm–Liouville spectral problem, Inverse Problems 7 (2011), Article ID 065011.

\bibitem{Dj}
Djakov, P.; Mityagin, B.N. Spectral gap asymptotics of one-dimensional Schr\"odinger operators with singular periodic potentials, Integral Transforms Spec. Funct. 20 (2009), no.~3--4, 265--273.

\bibitem{Eck15}
Eckhardt, J.; Gesztesy, F.; Nichols, R.; Sakhnovich, A.; Teschl, G. Inverse spectral problems for Schr\"odinger-type operators with distributional matrix-valued potentials, Diff. Integr. Equ. 28 (2015), no.~5/6, 505--522.

\bibitem{KVB24}
Kravchenko, V.V.; Vicente-Benitez, V.A.
Schr\"odinger equation with finitely many $\delta$-interactions: closed form, integral and series representations for solutions,
Anal. Math. Phys. 14 (2024), no.~5, Article ID 97.

\bibitem{Akt87}
Aktosun, T. Stability of the Marchenko inversion, Inverse Problems 3 (1987), 555--563.

\bibitem{MW05} 
Marletta, M.; Weikard, R. Weak stability for an inverse Sturm–Liouville problem with finite spectral data and complex potential, Inverse Problem 21 (2005), 1275--1290.

\bibitem{Bled12}
Beldsoe, M. Stability of the inverse resonance problem on the line, Inverse Problems 28 (2012), Article ID 105003.

\bibitem{BSY13}
Buterin, S.A.; Shieh, C.-T.; Yurko V.A. Inverse spectral problems for non-selfadjoint second-order differential operators with Dirichlet boundary conditions, Boundary Value Probl. 2013 (2013), Article ID 180.

\bibitem{BondButTr17}
Bondarenko, N.; Buterin, S. On a local solvability and stability of the inverse transmission eigenvalue problem, Inverse Problems 33 (2017), Article ID 115010.

\bibitem{BK19}
Buterin, S.; Kuznetsova, M. On Borg's method for non-selfadjoint Sturm-Liouville operators, Anal. Math. Phys. 9 (2019), 2133--2150.

\bibitem{MT19}
Mochizuki, K.; Trooshin, I. On conditional stability of inverse scattering problem on a lasso-shaped graph. In: Lindahl, K., Lindstr\"om, T., Rodino, L., Toft, J., Wahlberg, P. (eds) Analysis, Probability, Applications, and Computation. Trends in Mathematics, Birkh\"auser, Cham, 2019, 199--205.

\bibitem{XGY22}
Xu, X.-C.; Guo, Y.; Yang, C.-F. Stability of the inverse transmission eigenvalue problem for the Schr\"odinger operator with a radial potential, Lett. Math. Phys. 112 (2022), no. 4, Article ID 82.

\bibitem{GMXA23}
Guo, Y.; Ma, L.-J.; Xu, X.-C.; An, Q. Weak and strong stability of the inverse Sturm-Liouville problem, Math. Meth. Appl. Sci. 46 (2023), no. 14, 15684--15705.

\bibitem{BondR25}
Bondarenko, N.P. Uniform stability of the inverse problem for the non-self-adjoint Sturm-Liouville operator, arXiv:2409.16175.

\end{thebibliography}
\end{document}